\documentclass[11pt, twoside]{article}

\usepackage[T1]{fontenc}
\usepackage{amssymb,amsmath,amstext}
\usepackage{amsfonts}
\usepackage{amsthm}

\usepackage{fullpage} 
\usepackage{verbatim}

\theoremstyle{plain}
\newtheorem{thm}{Theorem}[section]
\newtheorem{cor}[thm]{Corollary}
\newtheorem{lma}[thm]{Lemma}

\newtheorem{ass}[thm]{Assumption}

\theoremstyle{definition}
\newtheorem{defi}[thm]{Definition}
\newtheorem{rmk}[thm]{Remark}
\newtheorem{exa}[thm]{Example}

\theoremstyle{remark}

\newtheorem*{claim}{Claim}

\numberwithin{equation}{section}

\newcommand{\cl}{\mathrm{cl}}

\newcommand{\mb}{\mathbf}

\newcommand{\es}{\emptyset}

\newcommand{\nts}{\negthickspace}

\newcommand{\mcA}{\mathcal{A}}
\newcommand{\mcB}{\mathcal{B}}
\newcommand{\mcC}{\mathcal{C}}
\newcommand{\mcD}{\mathcal{D}}

\newcommand{\mcF}{\mathcal{F}}
\newcommand{\mcG}{\mathcal{G}}
\newcommand{\mcH}{\mathcal{H}}

\newcommand{\mcM}{\mathcal{M}}
\newcommand{\mcN}{\mathcal{N}}

\newcommand{\mcP}{\mathcal{P}}

\newcommand{\mcS}{\mathcal{S}}

\newcommand{\mcU}{\mathcal{U}}
\newcommand{\mcV}{\mathcal{V}}

\newcommand{\mbC}{\mathbf{C}}

\newcommand{\mbG}{\mathbf{G}}

\newcommand{\mbK}{\mathbf{K}}
\newcommand{\mbS}{\mathbf{S}}

\newcommand{\mbX}{\mathbf{X}}

\newcommand{\mbbP}{\mathbb{P}}
\newcommand{\mbbN}{\mathbb{N}}
\newcommand{\mbbR}{\mathbb{R}}
\newcommand{\mbbZ}{\mathbb{Z}}

\newcommand{\K}{\textbf{K}}
\newcommand{\C}{\textbf{C}}
\newcommand{\M}{\mathcal{M}}
\newcommand{\A}{\mathcal{A}}
\newcommand{\B}{\mathcal{B}}

\newcommand{\G}{\mathcal{G}}
\newcommand{\V}{\mathcal{V}}
\newcommand{\W}{\mathcal{W}}

\newcommand{\Lrel}{L_{rel}}
\newcommand{\Vrel}{V_{rel}}
\newcommand{\Vcol}{V_{col}}

\newcommand{\reduct}{\upharpoonright}
\newcommand{\uhrc}{\nts \upharpoonright \nts}

\newcommand{\N}{\mathcal{N}}

\newcommand{\os}{\overset}
\newcommand{\us}{\underset}

\title{Random $l$-colourable structures with a pregeometry}

\author{Ove Ahlman and Vera Koponen}

\begin{document}

\maketitle

\begin{abstract} 
\noindent
We study finite $l$-colourable structures with an underlying pregeometry.
The probability measure that is used corresponds to a process of generating
such structures (with a given underlying pregeometry) by which colours are
first randomly assigned to all 1-dimensional subspaces and then relationships
are assigned in such a way that the colouring conditions are satisfied but
apart from this in a random way. We can then ask what the probability is that
the resulting structure, where we now forget the specific colouring of the
generating process, has a given property.
With this measure we get the following results:
\begin{enumerate}
\item A zero-one law.
\item The set of sentences with asymptotic probability 1 has an explicit
axiomatisation which is presented.
\item There is a formula $\xi(x,y)$ (not directly speaking about colours) 
such that, with asymptotic probability 1, the relation ``there is an $l$-colouring 
which assigns the same colour to $x$ and $y$'' is defined by $\xi(x,y)$.
\item With asymptotic probability 1, an $l$-colourable structure has a unique $l$-colouring
(up to permutation of the colours).
\end{enumerate}
\end{abstract}

\section{Introduction}\label{introduction}

\noindent
We begin with some background.
Let $l \geq 2$ be an integer.
Random $l$-colourable (undirected) graphs were studied by Kolaitis, Prömel and Rothschild
in \cite{KPR} as part of proving a zero-one law for $(l+1)$-clique-free graphs.  
They proved that random $l$-colorable graphs satisfies a (labelled) zero-one law, when the uniform probability
measure is used. In other words, if $\mbC_n$ denotes the set of undirected $l$-colourable graphs
with vertices $1, \ldots, n$, then, for every sentence $\varphi$ in a language with only a binary relation
symbol (besides the identity symbol), the proportion of graphs in $\mbC_n$ which satisfy $\varphi$
approaches either 0 or 1. They also showed that the proportion of graphs in $\mbC_n$ which have a
unique $l$-colouring (up to permuting the colours) approaches 1 as $n \to \infty$.
In \cite{KPR} its authors also proved the other statements labelled 1--4 in this paper's
abstract, when using the uniform probability measure on $\mbC_n$, although in case of~3 it is not
made explicit. 
This work was preceeded, and probably stimulated, by an article of Erd\"{o}s, Kleitman and
Rothschild \cite{EKR} in which it was proved that proportion of triangle-free graphs with
vertices $1, \ldots, n$ which are bipartite (2-colourable) approaches 1 as $n \to \infty$.

One can generalise $l$-colourings from structures with only binary relations to structures
with relations of any arity $r \geq 2$ by saying that a structure $\mcM$ is $l$-coloured if
the elements of $\mcM$ can be assigned colours from the set of colours $\{1, \ldots, l\}$
in such that if $\mcM \models R(a_1, \ldots, a_r)$ for some relation symbol $R$, then
$\{a_1, \ldots, a_r\}$ contains at least two elements with different colour.
Another way of generalising the notion of $l$-colouring for graphs, giving the notion
of {\em strong} $l$-colouring, is to require that if $\mcM \models R(a_1, \ldots, a_r)$
then $i \neq j$ implies that $a_i$ and $a_j$ have different colours.
(If the language has only binary relation symbols then there is no difference between
the two notions of $l$-colouring.)

In \cite{Kop12}, Koponen proved that, for every finite relational langauge,
if $\mbC_n$ is the set of $l$-colourable structures with universe $\{1, \ldots, n\}$, 
then the statements 1--4 from the
abstract hold, for the dimension conditional probability measure, as well as for the
uniform probability measure, on $\mbC_n$. The same results hold if we instead consider
{\em strongly} $l$-colourable structures. Moreover, the results still hold, for both types of
colourings and both probability measures, if we insist that some of the relation
symbols are always interpreted as irreflexive and symmetric relations.
A consequence of the zero-one law for (strongly) $l$-colourable structures is that
if, for some finite relational vocabulary and for each positive $n \in \mbbN$, $\mbK_n$ is a set of structures 
with universe $\{1, \ldots, n\}$ containing every $l$-colourable structure with that
universe, and the probability, using either the dimension
conditional measure or the uniform measure, that a random member of $\mbK_n$
is (strongly) $l$-colourable approaches 1 as $n \to \infty$, then $\mbK_n$ has a zero-one
law for the corresponding measure; i.e. for every sentence $\varphi$, the probability that
$\varphi$ is true in a random $\mcM \in \mbK_n$ approaches either 0 or 1 as $n \to \infty$.
In \cite{PS}, Person and Schacht proved that if $\mcF$ denotes the Fano plane as 
a 3-hypergraph (so $\mcF$ has seven elements and seven 3-hyperedges such that every
pair of distinct elements are contained in a unique 3-hyperedge) and
$\mbK_n$ is the set of $\mcF$-free 3-hypergraphs with universe (vertex set) $\{1, \ldots, n\}$,
then the proportion of hypergraphs in $\mbK_n$ which are 2-colourable approaches 1 as $n \to \infty$.
Since every 2-colourable 3-hypergraph is $\mcF$-free, 
it follows the $\mcF$-free 3-hypergraphs satisfy a zero-one law if we use the uniform probability measure.
As another example, Balogh and Mubayi \cite{BM} have proved that if $\mcH$ denotes the 
hypergraph with vertices $1, 2, 3, 4, 5$ and 3-hyperedges $\{1,2,3\}$, $\{1,2,4\}$
and $\{3,4,5\}$ and if $\mbK_n$ denotes the set of $\mcH$-free 3-hypergraphs with universe $\{1, \ldots, n\}$,
then the proportion of hypergraphs in $\mbK_n$ which are strongly 3-colourable approaches~1 as $n \to \infty$.
Since every strongly 3-colourable 3-hypergraph is $\mcH$-free it follows that 
$\mcH$-free 3-hypergraphs satisfy a zero-one law.

In the present article we generalise the work in \cite{Kop12} to the context of 
(strongly) $l$-colourable structures with an underlying (combinatorial) pregeometry,
also called matroid. 
Roughly speaking, a structure $\mcM$ with a pregeometry will be called $l$-colourable if
its 1-dimensional subspaces (i.e. closed subsets of $M$) can be assigned colours from
$l$ given colours in such a way that if $R$ is a relation symbol and
$\mcM \models R(a_1, \ldots, a_r)$, then there are $i$ and $j$ such that the subspaces spanned
by $a_i$ and by $a_j$, respectively, have different colours.
A structure $\mcM$ will be called {\em strongly} $l$-colourable if its 1-dimensional
subspaces can be assigned colours from $l$ given colours in such a way that
if $\mcM \models R(a_1, \ldots, a_r)$, then any two distinct 1-dimensional subspaces
that are included in the closure of $\{a_1, \ldots, a_r\}$ have different colours.
The main motivation for this generalisation is to understand how the combinatorics
of colourings work out if the elements of a structure are related to each other 
in a ``geometrical way'', where in particular, the role of cardinality is taken over by dimension. 
The main examples of pregeometries for which the results of this article apply
are vector spaces, projective spaces and affine spaces over some fixed finite field.
Another motivation is the fact that pregeometries have played an important
role in the study of infinite models and one may ask to what extent the notion of pregeometry
can be combined with the study of asymptotic properties of finite structures.
In \cite{Kop12} a framework for studying asymptotic properties of finite structures with an
underlying pregeometry was presented. Here we work within that framework, but since
we only consider (strongly) $l$-colourable structures some notions from \cite{Kop12} become
simpler here.

We now give rough explanations of the notions that will be involved and the main results.
Precise definitions are given in Section~\ref{preliminaries}.
We fix an integer $l \geq 2$.
$L_{pre}$ denotes a first-order language and for every $n \in \mbbN$, $\mcG_n$ is an $L_{pre}$-structure
such that $(G_n, \cl_{\mcG_n})$ is a pregeometry (Definition~\ref{definition of pregeometry}) where the 
closure operator $\cl_{\mcG_n}$ is definable by $L_{pre}$-formulas (in a sense given by Definition~\ref{spredef} and Assumption~\ref{assumptions on languages}). 
We will consider the property `polynomial $k$-saturation' 
(Definition~\ref{definition of polynomial k-saturation for classes}) of the enumerated set
$\mbG = \{\mcG_n : n \in \mbbN\}$. 
From Assumption~\ref{assumptions on languages} it follows that the dimension of $G_n$ 
approaches infinity as $n$ tends to infinity.
The language $L_{rel}$ (from Assumption~\ref{assumptions on languages}) includes $L_{pre}$ and has, in addition,
finitely many new relation symbols, all of arity at least 2.
By $\mbC_n$ we denote the set of all $L_{rel}$-structures $\mcM$ such that
$\mcM \uhrc L_{pre} = \mcG_n$ and $\mcM$ is $l$-colourable (Definition~\ref{coldef}).
By $\mbS_n$ we denote the set of all $L_{rel}$-structures $\mcM$ such that
$\mcM \uhrc L_{pre} = \mcG_n$ and $\mcM$ is strongly $l$-colourable (Definition~\ref{coldef}).
For every $n$, $\delta_n^\mbC$ denotes the probability measure given by 
Definition~\ref{definition of probability measures},
which means, roughly speaking, that if $\mbX \subseteq \mbC_n$, then $\delta_n^\mbC(\mbX)$
is the probability that $\mcM \in \mbC_n$ belongs to $\mbX$ if $\mcM$ generated by
the following procedure: first randomly assign $l$ colours to the 1-dimensional subspaces
of $M$, then, for every relation symbol $R$ that belongs to the vocabulary of $L_{rel}$ but
not to the vocabulary of $L_{pre}$, choose an interpretation of $R$ randomly 
from all possibilities of interpretations $R^\mcM$ 
such that the previous assignment of colours is an $l$-colouring of the resulting structure,
and finally forget the colour assignment, leaving us with an $L_{rel}$-structure.
The probability measure $\delta_n^\mbS$ on $\mbS_n$ is defined similarly 
(Definition~\ref{definition of probability measures}).
If $\varphi$ is an $L_{rel}$-sentence then 
$\delta_n^\mbC(\varphi) = \delta_n^\mbC\big(\{\mcM \in \mbC_n : \mcM \models \varphi\}\big)$
and similarly for $\delta_n^\mbS(\varphi)$.

\begin{thm}\label{zero-one law}
Suppose that Assumption~\ref{assumptions on languages} holds and that 
$\mbG = \{\mcG_n : n \in \mbbN\}$ is polynomially $k$-saturated for every $k \in \mbbN$.
Then, for every $L_{rel}$-sentence $\varphi$,
$\delta_n^\mbC(\varphi)$ approaches either 0 or 1, and
$\delta_n^\mbS(\varphi)$ approaches either 0 or 1, as $n \to \infty$.
\end{thm}

\noindent
If $F$ is a field and $G$ is the set of vectors of a vector space or of an affine space over $F$, or if $G$ is the
set of lines of a projective space over $F$, then $(G, \cl)$ where $\cl$ is the linear closure operator,
affine closure operator, or projective closure operator, respectively, forms a pregeometry
(see for example \cite{Oxl} or \cite{Mar}).

\begin{thm}\label{definability of colourings}
Suppose that the conditions of Assumption~\ref{assumptions on languages} hold
and that for some finite field $F$ one of the following three cases holds for every $n \in \mbbN$:
$G_n$ is an (a) $n$-dimensional vector space, or (b) $n$-dimensional affine space, or
(c) $n$-dimensional projective space, over $F$, and $\cl_{\mcG_n}$ is the linear, affine or 
projective closure operator on $G_n$, respectively.
Moreover, assume that $L_{pre}$ is the generic language $L_{gen}$ from 
Example~\ref{generic example of first-order pregeometry},
with the intepretations of symbols given in that example.\\
(i) There is an $L_{rel}$-formula $\xi(x,y)$ such that the $\delta_n^\mbC$-probability that
the following holds for $\mcM \in \mbC_n$ approaches 1 as $n \to \infty$:
\begin{itemize}
\item[] For all $a,b \in M - \cl_\mcM(\es)$, $\mcM \models \xi(a,b)$ if and only if
every $l$-colouring of $\mcM$ gives $a$ and $b$ the same colour.
\end{itemize}
(ii) $\lim_{n\to\infty} 
\delta_n^\mbC\big(\{\mcM \in \mbC_n : \text{ $\mcM$ has a unique $l$-colouring}\}\big) = 1$.\\
(iii) $\lim_{n\to\infty} 
\delta_n^\mbC\big(\{\mcM \in \mbC_n : \text{ $\mcM$ is not $l'$-colourable if $l' < l$}\}\big) = 1$.\\
(iv) The set $\{\varphi \in L_{rel} : \lim_{n\to\infty} \delta_n^\mbC(\varphi) = 1 \}$
forms a countably categorical theory which
can be explicitly axiomatised (as in Section~\ref{cc}) by $L_{rel}$-sentences of the
form $\forall \bar{x}\exists \bar{y} \psi(\bar{x}, \bar{y})$ where $\psi$ is quantifier-free, 
mainly in terms of what we call {\rm $l$-colour compatible extension axioms},
which involve the formula $\xi(x,y)$ from part~(i).\\
(v) The statements~(i)--(iv) hold if we assume that, for each $n \in \mbbN$,
$G_n$ is an $n$-dimensional vector space over $F$, $\cl_{\mcG_n}$ is the linear closure
operator on $G_n$ and that $L_{pre}$ is the language $L_F$
from Example~\ref{example of vector space over finite field}, with the interpretation of symbols from that example.
\end{thm}

\noindent
The assumptions of Theorem~\ref{definability of colourings} imply the assumptions
of Theorem~\ref{definability of strong colourings}, which is explained in 
Example~\ref{examples of polynomially k-saturated pregeometries}.
That is, when dealing with strongly $l$-colourable structures, the assumptions on
the underlying pregeometries can be weaker.
By a subspace of a pregeometry we mean a closed set with respect to the given closure operator
(Definition~\ref{spredef}).

\begin{thm}\label{definability of strong colourings}
Suppose that the conditions of Assumption~\ref{assumptions on languages} hold
and that $\mbG$ is polynomially $k$-saturated for every $k \in \mbbN$.
Also assume that for every $n \in \mbbN$, every 2-dimensional subspace of $\mcG_n$ has at most~$l$ different
1-dimensional subspaces.\\
(i) There is an $L_{rel}$-formula $\xi(x,y)$ such that the $\delta_n^\mbS$-probability that
the following holds for $\mcM \in \mbS_n$ approaches 1 as $n \to \infty$:
\begin{itemize}
\item[] For all $a,b \in M - \cl_\mcM(\es)$, $\mcM \models \xi(a,b)$ if and only if
every $l$-colouring of $\mcM$ gives $a$ and $b$ the same colour.
\end{itemize}
(ii) $\lim_{n\to\infty} 
\delta_n^\mbS\big(\{\mcM \in \mbS_n : \text{ $\mcM$ has a unique strong $l$-colouring}\}\big) = 1$.\\
(iii) $\lim_{n\to\infty} 
\delta_n^\mbS\big(\{\mcM \in \mbS_n : \text{ $\mcM$ is not strongly $l'$-colourable if $l' < l$}\}\big) = 1$.\\
(iv) Suppose, moreover, that the formulas of $L_{pre}$ which,
according to Assumption~\ref{assumptions on languages}, define the pregeometry
$\mbG = \{\mcG_n : n \in \mbbN\}$ are quantifier-free.
Then the set $\{\varphi \in L_{rel} : \lim_{n\to\infty} \delta_n^\mbS(\varphi) = 1 \}$
forms a countably categorical theory which
can be explicitly axiomatised (as in Section~\ref{cc}) by $L_{rel}$-sentences of the
form $\forall \bar{x}\exists \bar{y} \psi(\bar{x}, \bar{y})$ where $\psi$ is quantifier-free, 
mainly in terms of what we call {\rm $l$-colour compatible extension axioms},
which involve the formula $\xi(x,y)$ from part~(i).
\end{thm}

\noindent
It turns out that Theorem~\ref{zero-one law} follows rather straightforwardly from
Theorem~7.32 in \cite{Kop12} when we have proved
Lemma~\ref{acceptance of substitutions by l-colourable structures} below.
However, Theorem~\ref{zero-one law} in itself does not
give information about which sentences have asymptotic probability 1 (or 0), or
about properties of the theory consisting of those sentences which have asymptotic probability 1.
Neither does it tell us anything about typical properties of large (strongly) $l$-colourable structures.
In order to prove part~(i) of Theorems~\ref{definability of colourings}
and~\ref{definability of strong colourings}, which give information of this kind, 
we treat $l$-colourable structures and 
strongly $l$-colourable structures separately and need to add some assumption(s). 
The case of strong $l$-colourings is the easier one and is treated in Section~\ref{sc};
that is, most of the argument leading to part~(i) of Theorem~\ref{definability of strong colourings} is 
carried out in Section~\ref{sc}.
The main part of the proof of~(i) of Theorem~\ref{definability of colourings}, dealing with (not
necessarily strong) $l$-colourings, is carried out in Section~\ref{wc} where we
use a theorem from structural Ramsey theory by
Graham, Leeb and Rothschild \cite{GLR}.

Once we have established part~(i) of Theorems~\ref{definability of colourings}
and~\ref{definability of strong colourings}, which, as said above, is done separately,
parts~(ii)--(iv) (and (v) of Theorem~\ref{definability of colourings})
can be proved in a uniform way, that is, it is no longer necessary to distinguish 
between $l$-colourable structures and strongly $l$-colourable structures.
This is done in Section~\ref{cc}.
It is possible to read Section~\ref{cc} directly after Section~\ref{preliminaries}
and then consider the details of definability of colorings in Sections~\ref{sc} and~\ref{wc}, 
which are independent of each other.

The theorems above generalise the results of Section~9 of \cite{Kop12} to the situation
when a nontrivial pregeometry (subject to certain conditions) is present.
In other words, if the closure of a set $A$ is always $A$ (so every set is closed) and
we let $L_{pre}$ be the language whose vocabulary contains only the identity symbol `$=$',
and, for every $n \in \mbbN$, $\mcG_n$ is the unique (under these assumtions) $L_{pre}$-structure
with universe $\{1, \ldots, n+1\}$, then~(i)--(iv) of Theorems~\ref{definability of colourings}
and~\ref{definability of strong colourings} hold by results in Section~9 of \cite{Kop12}.
Theorem~\ref{zero-one law} includes this case, without reformulation.

\begin{rmk}\label{remark about symmetry of relations}
One may want to consider only $L_{rel}$-structures in which certain relation symbols from
the vocabulary of $L_{rel}$ are always interpreted as irreflexive and symmetric relations
(see beginning of Section~\ref{preliminaries}).
Theorems~\ref{zero-one law} and~\ref{definability of colourings} hold with exactly the same
proofs also in this situation. This claim uses that all results of \cite{Kop12} (see Remark~2.1
of that article) hold whether or not one assumes that certain relation symbols are always
interpreted as irreflexive and symmetric relations.
If a technical assumption is added, explained in Remark~\ref{remark about irreflexive symmetric relations},
then Theorem~\ref{definability of strong colourings} also holds in the context when some relation symbols are
always interpreted as irreflexive and symmetric relations.
\end{rmk}

\begin{rmk}
In \cite{Kop12}, results corresponding to Theorems~\ref{zero-one law}--\ref{definability of strong colourings},
in the case of trivial pregeometries (i.e. when every set is closed), where proved also for the
{\em uniform probability measure}. The proof used the fact, proved in Section~10 of \cite{Kop12},
that, when the pregeometries considered are trivial, then the probability, with the uniform
probability measure, that a random (strongly) $l$-colourable structure with $n$ elements has
an $l$-colouring with relatively even distribution of colours, approaches 1 as $n \to \infty$.
We believe that the same is true in the context of Theorems~\ref{definability of colourings}
and~\ref{definability of strong colourings} above, by proofs analogous to those
in Section~10 of \cite{Kop12}. But when the underlying pregeometries are no longer assumed
to be trivial, then this condition alone seems to be insufficient for proving analogoues of
Theorems~\ref{zero-one law}--\ref{definability of strong colourings} if $\delta_n^\mbC$ 
is replaced by the uniform probability measure on $\mbC_n$ and $\delta_n^\mbS$ is replaced
by the uniform probability measure on $\mbS_n$.
In other words, it appears to be a more difficult task to transfer the results of this
article to the uniform probability measure (if possible at all) than was the case in \cite{Kop12}.
\end{rmk}

\noindent
This article ends with a small errata to \cite{Kop12},
which makes explicit some assumptions, used implicity in Section~8 of \cite{Kop12},
but not stated explicitly in the places in Sections~7--8 of \cite{Kop12} where they are relevant.

\section{Pregeometries and (strongly) $l$-colourable structures}\label{preliminaries}

\noindent
The notation used here is more or less standard; see \cite{EF, Mar} for example.
The formal languages considered are always first-order and denoted $L$, often with a subscript.
Such $L$ denotes the set of first-order formulas over some vocabulary, also called signature, consisting
of constant-, function- and/or relation symbols.
First-order structures are denoted with calligraphic letters $\mcA, \mcB, \ldots, \mcM, \mcN, \ldots$,
and their universes with the corresponding non-calligraphic letters $A, B, \ldots, M, N, \ldots$.
If the vocabulary of a language $L$ has no constant or function symbols, then we allow an $L$-structure
to have an empty universe.
Finite sequences/tuples of objects, usually elements from structures or variables,
are denoted with $\bar{a}$, $\bar{x}$, etc. 
By $\bar{a} \in A$ we mean that every element of the sequence $\bar{a}$ belongs to the set $A$,
and $|A|$ denotes the cardinality of $A$.
A function $f : M \to N$ is called an {\bf \em embedding} of $\mcM$ into $\mcN$ if,
for every constant symbol $c$, $f(c^\mcM) = c^\mcN$, for every function symbol
$g$ and tuple $(a_1, \ldots, a_r) \in M^r$ where $r$ is the arity of $g$, 
$g^\mcN(f(a_1), \ldots, f(a_r)) = f(g^\mcM(a_1, \ldots, a_r))$, and for every relation symbol
$R$ and tuple $(a_1, \ldots, a_r) \in M^r$ where $r$ is the arity of $R$,
$\mcM \models R(a_1, \ldots, a_r)$ $\Longleftrightarrow$ $\mcN \models R(f(a_1), \ldots, f(a_r))$.
It follows that an {\bf \em isomorphism} from $\mcM$ to $\mcN$ is the same as a
surjective embedding from $\mcM$ to $\mcN$.
Suppose that $L'$ is a language whose vocabulary is included in the vocabulary of $L$.
For any $L$-structure $\mcM$, by $\mcM \uhrc L'$ we denote the reduct of $\mcM$ to $L'$.
If $\mcM$ is an $L$-structure and $A \subseteq M$, then $\mcM \uhrc A$ denotes the
substructure of $\mcM$ which is generated by the set $A$, that is, $\mcM \uhrc A$ is
the unique substructure $\mcN$ of $\mcM$
such that $A \subseteq N \subseteq M$ and if $\mcN' \subseteq \mcM$ and $A \subseteq N' \subseteq M$,
then $\mcN \subseteq \mcN'$. A third meaning of the symbol `$\upharpoonright$' with respect to structures is
given by Definition~\ref{definition of d-dimensional reduct}.
Suppose that $A$ is a set, $n \geq 2$ and $R \subseteq A^n$ an $n$-ary relation on $A$.
We call $R$ {\bf \em irreflexive} if $(a_1, \ldots, a_n) \in R$ implies that $a_i \neq a_j$ whenever
$i \neq j$. We call $R$ {\bf \em symmetric} if $(a_1, \ldots, a_n) \in R$ implies that
$(\pi(a_1), \ldots, \pi(a_r)) \in R$ for every permutation $\pi$ of $\{a_1, \ldots, a_n\}$.
For any set $A$, $\mcP(A)$ denotes the power set of $A$.
A usual, we call a formula {\bf \em existential} if it has the form 
\[\exists y_1, \ldots, y_m \varphi(x_1, \ldots, x_k, y_1, \ldots, y_m)\]
where $\varphi$ is quantifier free.

\begin{defi}\label{definition of pregeometry}
We say that $(A,\cl)$, with $\cl : \mathcal{P}(A)\rightarrow\mathcal{P}(A)$ is a {\bf \em pregeometry} 
(also called {\em matroid}) if it satisfies the following for all $X, Y \subseteq A$:
\begin{enumerate}
\item (Reflexivity)  $X\subseteq \cl(X)$.
\item (Monotonicity) $Y\subseteq \cl(X) \Rightarrow \cl(Y)\subseteq \cl(X)$.
\item (Exchange property) If $a,b\in A$ then $a\in \cl(X\cup\{b\}) - \cl(X)\Rightarrow b\in \cl(X\cup\{a\})$.
\item (Finite Character) $\cl(X) = \bigcup \{\cl(X_0) : X_0\subseteq X \text{ and } |X_0| \text{ is finite}\}$.
\end{enumerate}
\end{defi}

\noindent
If $X,Y \subseteq A$ then we say that $X$ is {\bf \em independent from $Y$} if 
$\cl(X) \cap \cl(Y) = \cl(\es)$.
From the exchange property it follows that $X$ is indepependent from $Y$ if and only
if $Y$ is independent from $X$ (symmetry of independence).
We will often write $\cl(a_1, \ldots, a_n)$ instead of $\cl(\{a_1, \ldots, a_n\})$ and
say `$a$ is independent from $b$' instead of `$\{a\}$ is independent from $\{b\}$ over $\es$'. 
We say that a set $X$ is {\bf \em independent} if for, each $a\in X$, we have that $\{a\}$ is independent from $X-\{a\}$.  
We say that a set $X\subseteq A$ is {\bf \em closed} (in $(A, \cl)$) if $\cl(X)=X$. 
For $X \subseteq A$, the dimension of $X$ is defined as 
$\dim(X) = \inf\{|Y| : Y \subseteq X \text{ and } X \subseteq \cl(Y)\}$.
For more about pregeometries the reader is refered to \cite{Mar, Oxl} for example.
We will use the following lemma, which has probably been proved somewhere, but for
the sake of completeness we give a proof of it here. 

\begin{lma}~\label{pregl1}
Let $\A=(A,\cl)$ be a pregeometry. 
If $\{a,v_1,...,v_m,w_1,...,w_n\}\subseteq A$ is an independent set then $\cl(a,v_1,...,v_m) \cap \cl(a,w_1,...,w_n) = \cl(a)$
\end{lma}

\begin{proof}
Suppose that $\{a,v_1,...,v_m,w_1,...,w_n\}\subseteq A$ is an independent set.
By reflexivity $a\in \cl(a,v_1,...,v_m) \cap \cl(a,w_1,...,w_n)$ and so by monotonicity 
\[ \cl(a) \ \subseteq \ \cl(a,v_1,...,v_m)\cap \cl(a,w_1,...,w_n).\]
For the opposite direction we assume that $x\in \cl(a,v_1,...,v_m)\cap \cl(a,w_1,...,w_{n})$ and use induction over $n$ to prove that $x\in \cl(a)$. 

{\em Base case}: If $n=0$ then $\cl(a,w_1,...,w_n) = \cl(a)$ so, as $x \in \cl(a)$, we are done.

{\em Induction step}: Suppose that $x\in \cl(a,v_1,...,v_m)\cap \cl(a,w_1,...,w_{n+1})$, so we have two cases to consider:
\begin{align*}
&\text{either } \ &x \ \in \  \big(\cl(a,v_1,...,v_m) \ \cap \ \cl(a,w_1,...,w_{n+1})\big) \\ 
& &- \ 
\big(\cl(a,v_1,...,v_m) \ \cap \ \cl(a,w_1,...,w_{n})\big),\\
&\text{or } &x \ \in \  \cl(a,v_1,...,v_m) \ \cap \ \cl(a,w_1,...,w_{n}).
\end{align*}
In the first case we get the consequence that $x\in \cl(a,w_1,...,w_{n+1}) - \cl(a,w_1,...,w_n)$ and hence by the exchange property we get that $w_{n+1} \in \cl(a,w_1,...,w_n,x)$. We already know that $x\in \cl(a,v_1,...,v_m)$ and
by also using the assumption that $\{a,v_1,...,v_m,w_1,...,w_n\}$ is independent we get that
\[1+m+n = \dim(a,v_1,...,v_m,w_1,...,w_n) = \dim(a,v_1,...,v_m,w_1,...,w_n,x) =\]\[ \dim(a,v_1,...,v_m,w_1,...,w_n,w_{n+1},x) = \dim(a,v_1,...,v_m,w_1,...,w_{n+1}) = 1+m+n+1,\]
so $1 + m + n = 1 + m + n + 1$, a contradiction.
Hence, 
\[x \ \in \ \cl(a,v_1,...,v_m) \ \cap \ \cl(a,w_1,...,w_{n}),\]
so by the induction hypothesis we get that $x \in \cl(a)$.

By induction we conclude that $\cl(a,v_1,...,v_m) \ \cap \ \cl(a,w_1,...,w_{n}) \subseteq \cl(a)$ holds
for all $n$, which finishes the proof.
\end{proof}

\noindent 
We will consider first-order structures $\mcM$ for which there is
a closure operator $\cl$ on $M$ such that $(M, \cl)$ is a pregeometry and, for each $n$, the relation 
$x_{n+1} \in \cl(x_1, \ldots, x_n)$ is definable by a first-order formula without parameters.
More precisely, we have the following definition.

\begin{defi}~\label{spredef}
(i) We say that an $L$-structure $\A$ is a {\bf \em pregeometry} if there are $L$-formulas
$\theta_n(x_1, \ldots, x_{n+1})$, for all $n \in \mbbN$, such that if the operator 
$\cl_{\mcA} : \mcP(A) \to \mcP(A)$ is defined by~(a) and~(b) below, then $(A, \cl_{\mcA})$ is a pregeometry:
\begin{itemize}
\item[(a)] For every $n \in \mbbN$, every sequence $b_1, \ldots, b_n \in A$ and every $a \in A$, 
\[a \in \cl_{\mcA}(b_1, \ldots, b_n) \Longleftrightarrow
\mcA \models \theta_n(b_1, \ldots, b_n, a).\]
\item[(b)] For every $B \subseteq A$ and every $a \in A$, $a \in \cl_{\mcA}(B)$
if and only if $a \in \cl_{\mcA}(b_1, \ldots, b_n)$ for some $b_1, \ldots, b_n \in B$.
\end{itemize}
(ii) Suppose that $\A$ is a pregeometry in the sense of the above definition.
Then, for every $B \subseteq A$,
$\dim_\A(B)$ denotes the {\bf \em dimension} of $B$ with respect to the closure operator $\cl_{\mcA}$.
In other words, $\dim_{\mcA}(B) = \min\{|B'| : B' \subseteq B \text{ and } \cl_{\mcA}(B') \supseteq B\}$. 
We sometimes abbreviate $\dim_\mcB(B)$ with $\dim(\mcB)$.
A closed subset of $A$ is also called a {\bf \em subspace} of $\mcA$.
A substructure $\B \subseteq \A$ is called {\bf \em closed} if its universe $B$ is closed in $(A, \cl_\mcA)$.\\
(iii) Suppose that $\mbG$ is a set of $L$-structures. 
We say that $\mbG$ is a {\bf \em pregeometry} if 
there are $L$-formulas $\theta_n(x_1, \ldots, x_{n+1})$, for all $n \in \mbbN$,
such that for each $\A \in \mbG$, $(A, \cl_{\mcA})$ is a pregeometry if 
$\cl_{\mcA}$ is defined by~(a) and~(b).
\end{defi}

\noindent
It may happen that for an $L$-structure $\mcA$ there are $L$-formulas
$\theta_n$ and $\theta'_n$, for $n \in \mbbN$, such that the 
sequence $\theta_n$, $n \in \mbbN$, defines a different pregeometry on $A$
(according to Definition~\ref{spredef}~(i)) than does the sequence $\theta'_n$, $n \in \mbbN$.
When we use these notions it will, however, be clear that we fix a
sequence of formulas $\theta_n$, $n \in \mbbN$, and the pregeometry that they
define on each structure from a given set, which will be denoted $\mbG$.

\begin{exa}\label{generic example of first-order pregeometry}
({\bf Generic example})
Every pregeometry $(A, \cl)$ can be viewed as a first-order structure $\mcA$ in the following way.
For every $n \in \mbbN$, let $P_n$ be an $(n+1)$-ary relation symbol and let
the vocabulary of $L_{gen}$ be $\{P_n : n \in \mbbN\}$.
For every $n \in \mbbN$ and every $(a_1, \ldots, a_{n+1}) \in A^{n+1}$,
let $(a_1, \ldots, a_{n+1}) \in (P_n)^{\mcA}$ if and only if 
$a_{n+1} \in \cl(a_1, \ldots, a_n)$.
Then $\mcA$ is a pregeometry in the sense of Definition~\ref{spredef}~(i)
and $\cl = \cl_{\mcA}$.
It follows that every set of pregeometries $\mbG$, viewed as $L_{gen}$-structures
is a pregeometry in the sense of Definition~\ref{spredef}~(iii).
\end{exa}

\begin{exa}\label{example of trivial geometries}
({\bf Trivial pregeometries})
If $A$ is a set and $\cl(B) = B$ for every $B \subseteq A$,
then $(A, \cl)$ is a pregeometry, called a {\bf \em trivial preometry}.
Let $L_\es$ be the language with vocabulary $\es$, so $L_\es$ can only express
whether elements are identical or not. 
If, for $n > 0$, $\theta_n(x_1, \ldots, x_{n+1})$ denotes a formula
which expresses that ``$x_{n+1}$ is identical to one of $x_1, \ldots, x_n$'',
and $\theta_0(x_1)$ is some formula which can never be satisfied,
then every $L_\es$-structure is a pregeometry in the sense of
Definition~\ref{spredef}~(i).
Moreover, every set $\mbG$ of $L_\es$-structures is a pregeometry in the sense of 
Definition~\ref{spredef}~(iii).
\end{exa}

\begin{exa}\label{example of vector space over finite field}
({\bf Vector spaces over a finite field})
Let $F$ be a field. Let $L_F$ be the language with vocabulary
$\{0, +\} \cup \{f : f \in F\}$, where $0$ is a constant symbol, $+$ a binary
function symbol and each $f \in F$ represents a unary function symbol.
Every vector space over $F$ can be viewed as an $L_F$-structure by
interpreting $0$ as the zero vector, $+$ as vector addition and each
$f \in F$ as scalar multiplication by $f$.
Now add the assumption that $F$ is {\em finite}.
If, for every $n \in \mbbN$, $\theta_n(x_1, \ldots, x_{n+1})$ is an $L_F$-formula that
expresses that ``$x_{n+1}$ belongs to the linear span of $x_1, \ldots, x_n$'', then
every $F$-vector space $\mcV$, viewed as an $L_F$-structure,
is a pregeometry accordning to Definition~\ref{spredef}~(i).
In particular, every set $\mbG$ of vector spaces over a finite field $F$, viewed
as $L_F$-structures, is a pregeometry according to 
Definition~\ref{spredef}~(iii).
\end{exa}

\begin{defi}\label{definition of uniformly bounded}{\rm
We say that the pregeometry $\mbG = \{\mcG_n : n \in \mbbN\}$ is 
{\bf \em uniformly bounded} if there is a function $f : \mbbN \to \mbbN$
such that for every $n \in \mbbN$ and every $X \subseteq G_n$,
$\big|\cl_{\mcG_n}(X)\big| \leq f\big(\dim_{\mcG_n}(X)\big)$.
}\end{defi}

\begin{exa}\label{uniform boundedness of vector spaces etc}
({\bf Vector space pregeometries})
Let $\mbG = \{\mcG_n : n \in \mbbN\}$ is a pregeometry.
Suppose that, for every $n \in \mbbN$, 
$(G_n, \cl_{\mcG})$ is isomorphic (as a pregeometry) with $(V_n, \cl_{\mcV_n})$
where each $\mcV_n$ is a vector space of dimension $n$ over a (fixed) finite field $F$ and $\cl_{\mcV_n}$  
is linear span in $\mcV_n$. Then $\mbG = \{\mcG_n : n \in \mbbN\}$ is uniformly bounded.
We get the same conclusion if, instead, each $\mcV_n$ is a projective
space over $F$ with dimension $n$, or if each $\mcV_n$ is an affine space over $F$ 
with dimension $n$.
\end{exa}

\begin{exa}\label{example of linear closure in R-n}
({\bf Sub-pregeometries of $\mbbR^n$})
Let $\cl_n$ denote the linear closure operator in $\mbbR^n$.
It is straightforward to verify that whenever $X_n \subseteq \mbbR^n$ and
$\cl'_n$ is defined by $\cl'_n(A) = \cl_n(A) \cap X_n$ for every $A \subseteq X_n$,
then $(X_n, \cl'_n)$ is a pregeometry.
For every positive integer $n$ choose {\em finite} $X_n \subseteq \mbbR^n$
and, for all $n \in \mbbN$, let $\mcG_n = (X_{n+1}, \cl'_{n+1})$.
Let $L_{gen}$ be the language from Example~\ref{generic example of first-order pregeometry}.
Then each $\mcG_n$ can be viewed as a first-order structure in the way explained
in that example.
It follows that $\mbG = \{\mcG_n : n \in \mbbN\}$ is a pregeometry in the sense of 
Definition~\ref{spredef}~(iii).
Suppose that, in addition, the choice of each $X_n$ is made in such a way that for every $k > 0$
there is $m_k$ such that
if $n > 0$ and $a_1, \ldots, a_k \in X_n$, 
then $|\cl(a_1, \ldots, a_k) \cap X_n| \leq m_k$.
Then $\mbG$ is uniformly bounded.
\end{exa}

\begin{defi}\label{definition of polynomial k-saturation for classes}{\rm
Let $k \in \mbbN$.
We say that the pregeometry  $\mbG = \{\mcG_n : n \in \mbbN\}$ is {\bf \em polynomially $k$-saturated}
	if there are a sequence of natural numbers $(\lambda_n : n \in \mbbN)$ with
	$\lim_{n \to \infty}\lambda_n = \infty$ and a polynomial
	$P(x)$ such that for every $n \in \mbbN$:
	\begin{itemize}
		\item[(1)] $\lambda_n \leq |G_n| \leq P(\lambda_n)$, and
		\item[(2)] whenever $\mcA$ is a closed substructure of $\mcG_n$ and there are $\mcG$ and
		$\mcB \supset \mcA$ such that $\mcA$ and $\mcB$ are closed substructures of $\mcG$,
		$\mcG$ is isomorphic with some member of $\mbG$ and 
		$\dim_{\mcG}(A) + 1 = \dim_{\mcG}(B) \leq k$, 
		then there are closed substructures $\mcB_i \subseteq \mcM$, for $i = 1, \ldots, \lambda_n$,
		such that $B_i \cap B_j = A$ if $i \neq j$, and each $\mcB_i$ is isomorphic with $\mcB$
		via an isomorphism that fixes $A$ pointwise.
	\end{itemize}
}\end{defi}

\begin{exa}\label{examples of polynomially k-saturated pregeometries}
(i) Let $L_{\es}$ be the ``empty'' language from Example~\ref{example of trivial geometries}.
It is straightforward to verify that if for every $n \in \mbbN$, $\mcG_n$ is the unique $L_{\es}$-structure 
with universe $\{1, \ldots, n+1\}$, then $\mbG$ is polynomially $k$-saturated
for every $k \in \mbbN$.\\
(ii) Let $F$ be a finite field and let $L = L_{gen}$ as in 
Example~\ref{generic example of first-order pregeometry} 
or $L = L_F$ as in
Example~\ref{example of vector space over finite field}.
For $n \in \mbbN$ let $\mcV_n$ be a vector space over $F$ of dimension $n$.
Each $\mcV_n$ gives rise to a pregeometry $(V_n, \cl_n)$ where $\cl_n$ is linear span,
and each $\mcV_n$ can be viewed as an $L$-structure, call it $\mcG_n$, as in any one
of the mentioned examples (depending on whether we take $L = L_{gen}$ or $L = L_F$).
Then the pregeometry $\mbG = \{\mcG_n : n \in \mbbN\}$ is
polynomially $k$-saturated for every $k \in \mbbN$.
This is explained in some more detail in~\cite{Kop12} and the proofs in
Section~3.2 of~\cite{Djo06a} translate to the present context.\\
(iii) Let $F$ be a finite field.
If $\mcG_n$, for $n \in \mbbN$, is instead the pregeometry obtained from a
projective space over $F$ with dimension $n$,
viewed as an $L_{gen}$-structure as in Example~\ref{generic example of first-order pregeometry},
then $\mbG = \{\mcG_n : n \in \mbbN\}$ is polynomially $k$-saturated for every $k \in \mbbN$.
The same holds if `projective space' is replaced with `affine space'.
These facts are proved are proved in a slightly different context Section~3.2 of~\cite{Djo06a},
but the proofs there translate straightforwardly to the present context.
\end{exa}

\begin{ass}\label{assumptions on languages}{\rm
We now fix some notation and assumptions for the rest of the paper.
\begin{itemize}
\item[(1)] Let $l \geq 2$ be an integer, $P_1, \ldots, P_l$ unary relation symbols and let 
$V_{col} = \{P_1, \ldots, P_l\}$.
The symbols $P_i$ represent colours.
Let $V_{rel}$ be a finite nonempty set of relation symbols all of which have arity at least $2$.
Let $\rho$ be the maximal arity among the relation symbols in $V_{rel}$.

\item[(2)] Let $L_{pre}$ be a language with vocabulary $V_{pre}$, which is disjoint from both $V_{col}$ and $V_{rel}$.
Suppose that $\mbG = \{\mcG_n : n \in \mbbN\}$ is a set of finite $L_{pre}$-structures
where $G_n$ is the universe of $\mcG_n$ and $\mbG$ is a pregeometry in the sense of Definition~\ref{spredef}~(iii).
Also, assume that the $L_{pre}$-formulas $\theta_n(x_1, \ldots, x_{n+1})$, $n \in \mbbN$, define the
pregeometry according to Definition~\ref{spredef}.

\item[(3)] Let $L_{col}$ be the language with vocabulary $V_{pre} \cup V_{col}$,
let $L_{rel}$ be the language with vocuabulary $V_{pre} \cup V_{rel}$
and let $L$ be the language with vocabulary $V_{pre} \cup V_{col} \cup V_{rel}$.

\item[(4)] $\mbG$ is uniformly bounded and, for every $n \in \mbbN$, if $A \subseteq G_n$
is closed (with respect to $\cl_{\mcG_n}$) then $A$ is the universe of a substructure of $\mcG_n$
(or equivalently, $A$ contains all interpretations of constant symbols and is closed under interpretations
of function symbols, if such occur in the language).

\item[(5)] For every $n \in \mbbN$, if $\mcA$ is a closed substructure of $\mcG_n$
and $a_1, \ldots, a_{n+1} \in A$, then $a_{n+1} \in \cl_{\mcG_n}(a_1, \ldots, a_n)$
$\Longleftrightarrow$ $\mcA \models \theta_n(a_1, \ldots, a_{n+1})$.
In other words, the restriction of $\cl_{\mcG_n}$ to $A$ is definable in $\mcA$ by
the same formulas $\theta_n$.

\item[(6)] For every $n \in \mbbN$, if $\mcA$ is a closed substructure of $\mcG_n$,
then there is $m$ such that $\mcA \cong \mcG_m$.
Also assume that $\lim_{n\to\infty} \dim(\mcG_n) = \infty$ and,
for every $n \in \mbbN$, $\mcG_n \uhrc \cl_{\mcG_n}(\es) \cong \mcG_0$.

\item[(7)] For every $n \in \mbbN$, there is a ``characteristic'' quantifier-free $L_{pre}$-formula \\
$\chi_{\mcG_n}(x_1, \ldots,x_{m_n})$ of $\mcG_n$,
where $m_n = |G_n|$, such that if $\mcA$ is an $L_{pre}$-structure in which the formulas $\theta_n$
define a pregeometry (according to Definition~\ref{spredef}) and
$\mcA \models \chi_{\mcG_n}(a_1, \ldots, a_s)$ for some enumeration $a_1, \ldots, a_s$ of $A$,
then $\mcA \cong \mcG_n$.
\end{itemize}
}\end{ass}

\begin{rmk}\label{remark on assumptions on pregeometry}
(i) If $\theta_n$ is quantifier free for every $n \in \mbbN$,
then~(5) holds. Note that in all examples above, it is possible to
let $\theta_n$ be quantifier free for every $n \in \mbbN$,
either by using using the ``generic'' language $L_{gen}$ from 
Example~\ref{generic example of first-order pregeometry}, or by using some
of the other languages mentioned in the examples.\\
(ii) Observe that by~(5), if $\mcA$ is a closed substructure of $\mcG_n$ then
the formulas $\theta_n$ define a pregeometry $(A, \cl_\mcA)$, according to Definition~\ref{spredef}, and for all 
$X \subseteq A$, $\cl_\mcA(X) = \cl_{\mcG_n}(X)$. 
By~(5)--(6), for every $k \in \mbbN$, there are only finitely many $L_{pre}$-structures $\mcA$, 
up to isomorphism, such that
for some $n$, $\mcA \subseteq \mcG_n$ and $\dim_{\mcG_n}(A) \leq k$. \\
(iii) Condition~(7) obviously holds if the vocabulary $V_{pre}$ is finite.
But we want to be able to consider languages with infinite vocabularies, such as the
langauge $L_{gen}$ from Example~\ref{generic example of first-order pregeometry}.
If we take $L_{pre} = L_{gen}$ with the same interpretations as in 
Example~\ref{generic example of first-order pregeometry} and (1)--(6) hold, then also (7) holds.
\end{rmk}

\begin{defi}\label{coldef}
(i) We say that an $L$-structure $\mcN$ is {\bf \em $l$-coloured} if there is an $L$-structure $\mcM$ such
that $\mcM \cong \mcN$, $\mcM \uhrc L_{pre} = \mcG_n$ for some $n \in \mbbN$ and
$\mcM$ satisfies the following four conditions:
\begin{itemize}
\item[(1)] For all $a \in M$, $\mcM \models P_1(a)\vee...\vee P_l(a)$ if and only if $a \notin \cl_{\mcG_n}(\es)$,
in other words, an element has a colour if and only if it does not belong to the closure of $\es$.
\item[(2)] If $R \in \Vrel$ has arity $m \geq 2$ and $a_1, \ldots, a_m \in \cl_{\mcG_n}(\es)$, 
then $\mcM \models \neg R(a_1, \ldots, a_m)$.
\item[(3)] For all $i,j \in \{1,...,l\}$ such that $i\not = j$ and all $a,b \in M - \cl_{\mcG_n}(\emptyset)$ such that $a\in \cl_{\mcG_n}(b)$ we have that $\M\models \neg (P_i(a) \wedge P_j(b))$, i.e. dependent elements not belonging to
the closure of $\es$ have the same colour.
\item[(4)] If $R\in \Vrel$ has arity $m\geq 2$ and $\M\models R(a_1,...,a_m)$ then there are 
$b,c \in \cl_{\mcG_n}(a_1,...,a_m) - \cl_{\mcG_n}(\es)$ such that for every $k\in\{1,...,l\}$ we have 
$\M\models \neg (P_k(b) \wedge P_k(c))$.
\end{itemize}
(ii) We say that $\N$ is {\bf \em strongly $l$-coloured} if there is an $L$-structure $\mcM$ such
that $\mcM \cong \mcN$, $\mcM \uhrc L_{pre} = \mcG_n$ for some $n \in \mbbN$ and
$\mcM$ satisfies~(1)--(4) above and (5) below:
\begin{enumerate}
\item[(5)] If $R\in \Vrel$ has arity $m\geq 2$ and $\M\models R(a_1,...,a_m)$, then for all 
$b,c \in \cl_{\mcG_n}(a_1,...,a_m) - \cl(\es)$ that are linearly independent ($b\notin \cl_{\mcG_n}(c)$) and every $k\in\{1,...,l\}$,  $\M\models \neg (P_k(b) \wedge P_k(c))$. 
\end{enumerate}
(iii) An $L_{rel}$-structure is called {\bf \em (strongly) $l$-colourable} if it can be
expanded to an $L$-structure that is (strongly) $l$-coloured.\\
(iv) For $n \in \mbbN$, let $\mbK_n$ denote the set of all $l$-coloured structures $\mcM$ such that 
$\mcM \uhrc L_{pre} = \mcG_n$
and let $\mb{SK}_n$ denote the set of all strongly $l$-coloured structures $\mcM$ such that
$\mcM \uhrc L_{pre} = \mcG_n$. Similarly, let $\mbC_n$ and $\mbS_n$ denote the set
of $l$-colourable, respectively, strongly $l$-colourable structures $\mcM$ such that
$\mcM \uhrc L_{pre} = \mcG_n$.
Finally, let $\mbK = \bigcup_{n\in\mbbN} \mbK_n$, $\mb{SK} = \bigcup_{n\in\mbbN} \mb{SK}_n$
$\mbC = \bigcup_{n\in\mbbN} \mbC_n$ and $\mbS = \bigcup_{n\in\mbbN} \mbS_n$
\end{defi} 

\noindent
It follows that if $\mcM$ is (strongly) $l$-colourable (or $l$-coloured) and 
all $a_1, \ldots, a_r \in M$ belong to the same 0- or 1-dimensional subspace,
then $\mcM \not\models R(a_1, \ldots, a_r)$.

\begin{rmk}\label{remark about definition of l-coloured structures}
(i) If we say that $\mcM$ is (strongly) $l$-coloured then it is presupposed that $\mcN$ is an
is an $L$-structure. If we say that $\mcM$ is (strongly) $l$-colourable then it is presupposed 
that $\mcM$ is an $L_{rel}$-structure.\\
(ii) From Definition~\ref{coldef} it follows that if $\mcM$ is (strongly) $l$-coloured 
or (strongly) $l$-colourable, then the formulas $\theta_n(x_1, \ldots, x_{n+1})$ 
from Assumption~\ref{assumptions on languages} define a pregeometry on $M$
according to Definition~\ref{spredef}. 
We always have this pregeometry in mind when speaking of the pregeometry of
an (strongly) $l$-coloured or (strongly) $l$-colourable structure.\\
(iii) From the definition of (strongly) $l$-coloured and (strongly) $l$-colourable structures
and Assumption~\ref{assumptions on languages}
it follows that if $\mcM$ is a (strongly) $l$-coloured, or (strongly) $l$-colourable,
structure, and $\mcA$ is a closed substructure of $\mcM$, then $\cl_\mcA(X) = \cl_\mcM(X)$
for every $X \subseteq A$. For this reason we will usually omit the subscripts `$\mcA$' and `$\mcM$'
and just write `$\cl$'.
Also note that from Assumption~\ref{assumptions on languages} it follows that there is
a unique (strongly) $l$-coloured/colourable structure of dimension 0.
\end{rmk}

\begin{defi}\label{definition of d-dimensional reduct}
Suppose that $\M$ is an $L$-structure. 
Let $d\in\mathbb{N}$. 
The {\bf \em $d$-dimensional reduct} of $\M$, denoted $\M\uhrc d$, is the unique $L$-structure satisfying the following three conditions:
\begin{itemize}
\item[(1)] $\M\uhrc d$ has the same universe as $\M$.
\item[(2)] Every symbol in $V_{pre}$ has the same interpretation in $\M \uhrc d$ as in $\M$.
\item[(3)] For each relation symbol $R \in V_{col} \cup V_{rel}$ and tuple $\bar{a} \in M$ of the 
corresponding arity,
\[ \bar a \in R^{\M\reduct d} \Leftrightarrow \dim_\M(\bar a) \leq d \text{ and } \bar a \in R^\M.\]
\end{itemize}
Let $\K_n \uhrc d = \{\M\uhrc d : \M\in\K_n\}$ and 
$\mb{SK}_n \uhrc d = \{\M \uhrc d : \M \in \mb{SK}_n\}$.
\end{defi}

\noindent Notice that if $\mcM$ is a (strongly) $l$-colourable structure and
$d$ is an integer such that no relation symbol in $\Vrel$ has higher arity than $d$,
then $\mcM \uhrc d = \mcM$.
We also have $\mbK_n \uhrc 0 = \{\mcG_n\} = \mb{SK}_n \uhrc 0$ for every $n$.
By the {\bf \em uniform probability measure} on a finite set $X$ we mean the probability
measure which gives every member of $X$ the same probability $1/|X|$.
Recall from Assumption~\ref{assumptions on languages} that $\rho$
is the highest arity that occurs among the relation symbols of $V_{rel}$, so $\rho \geq 2$.

\begin{defi}\label{definition of probability measures}
(i) For every $n \in \mbbN$ and every integer $0 \leq r \leq \rho$ we define a
probability measure $\mbbP_{n,r}$ on $\K_n \uhrc r$ by induction on $r$ as follows.
$\mathbb{P}_{n,0}$ is the uniform probability measure on $\K_n\uhrc 0$. 
For each $1\leq r\leq \rho$ and $\M\in \K_n\uhrc r$ we define
\[ \mathbb{P}_{n,r}(\M) = \frac{1}{|\{\M'\in \K_n\uhrc r: \M'\reduct r-1 = \M\uhrc r-1\}|} \cdot \mathbb{P}_{n,r-1}(\M\uhrc r-1).\]
(ii) We then define $\delta^\K_n = \mathbb{P}_{n,\rho}$ which we call the 
{\bf \em dimension conditional measure on} $\K_n = \K_n \uhrc \rho$. \\
(iii) The {\bf \em dimension conditional measure on} $\mb{SK}_n$, denoted $\delta^\mb{SK}_n$,
is defined in the same way, by replacing $\K_n$ with $\mb{SK}_n$ in part~(i) and then letting 
$\delta^\mb{SK}_n = \mbbP_{n,\rho}$.
\end{defi}

\begin{exa}
Let $L_{pre} = L_F$ as in Example~\ref{example of vector space over finite field} and let $F = \mbbZ_2$.
Suppose that $l = 2$, so $\Vcol = \{P_1, P_2\}$, and suppose that $\Vrel = \{R\}$ where $R$ is binary.
Let $\G_2=\mbbZ_2 \times \mbbZ_2$, that is, $\G_2$ is a 2-dimensional vector space over $\mbbZ_2$. 
From the assumptions that have been made it follows that $\K_2$ is the set of all $2$-coloured structures 
$\M$ such that $\M \uhrc L_{pre} = \G_2 = \mbbZ_2 \times \mbbZ_2$.
We have $|\K_2| = 26$,
so if $\M\in\K_2$ is the structure in which all non-zero vectors have colour $P_1$ 
and consequently $R^\M = \emptyset$, then with the uniform probability measure the probability of $\M$ is $1/26$.

If we want to calculate $\delta^{\K}_2(\M)$, where $\M$ is still the same structure, we first need to calculate $\mathbb{P}_{2,0}(\M \uhrc 0)$ which equals $1$, because $\mbbP_{2,0}$ is the uniform probability
on $\K_2 \uhrc 0$ which contains exactly one structure, namely $\G_2 = \mcM \uhrc 0$.
When we consider $\mathbb{P}_{2,1}(\M \uhrc 1)$ we look at structures in $\K_2\uhrc 1$, that is,
$\G_2$ with colours added. 
Since $|\K_2 \uhrc 1| = 8$ and the 0-dimensional reduct of every member of $\K_2 \uhrc 1$ is $\G_2$
it follows that
\[\mathbb{P}_{2,1}(\M \uhrc 1) \ = \ \frac{1}{|\{\M'\in \K_2\reduct 1: \M'\reduct 0 = \M\reduct 0\}|} \cdot \mathbb{P}_{2,0}(\M \uhrc 0)
= \frac{1}{8} \cdot 1 = \frac{1}{8}.\]
The last step, to calculate $\delta^\K_2(\M) = \mathbb{P}_{2,2}(\M)$ is easy, since the only structure in $\K_2\uhrc 2 = \K_2$ which has the same colouring as $\M$ is $\M$ itself.
Hence 
\[\delta^\K_2(\M) \ = \ \mathbb{P}_{2,2}(\M) \ = \ 
\frac{1}{|\{\M'\in \K_2\reduct 2: \M'\reduct 1 = \M\reduct 1\}|} \cdot 
\mathbb{P}_{2,1}(\M \uhrc 1)
= \frac{1}{1} \cdot \frac{1}{8} = \frac{1}{8}.\]
\end{exa}

\begin{rmk}\label{remark about dimension conditional measure}
We defined $\delta_n^\mbK$ and $\delta_n^{\mb{SK}}$ as we did in 
Definition~\ref{definition of probability measures}
because we are going to use results from~\cite{Kop12}.
But in the present (more specialised) context, $\delta_n^\mbK$ can be more simply
characterised as follows. For every $\mcM \in \mbK_n$ we have
\[ \delta_n^\mbK(\mcM) \ = \ 
\frac{1}{\big|\mbK_n \uhrc 1\big| \cdot \big|\{\mcM' \in \mbK_n : \mcM' \uhrc 1 = \mcM \uhrc 1\}\big|},\]
and similarly for $\delta_n^{\mb{SK}}$.
This is not difficult to prove, by the use of the definitions of $l$-coloured, and strongly $l$-coloured,
structures. Note that any given colouring of an $l$-coloured structure $\mcM \in \mbK_n$
has probability $1 / |\mbK_n \uhrc 1|$ with this measure.
\end{rmk}

\begin{defi}\label{extax}\label{definition of extension property}
Let $\mcM$ be an (strongly) $l$-coloured structure.\\
(i) Suppose that $\mcB$ is an (strongly) $l$-coloured structure and that
$\mcA$ is a closed substructure of $\mcB$, so $\mcA$ is also (strongly) $l$-coloured.
We say that $\mcM$ has the {\bf \em $\mcB/\mcA$-extension property} if 
whenever $\mcA'$ is a closed substructure of $\mcM$ and
$\sigma_{\mcA} : \mcA' \to \mcA$ is an isomorphism, then there are a closed substructure
$\mcB'$ of $\mcM$ such that $\mcA' \subset \mcB'$
and an isomorphism $\sigma_{\mcB} : \mcB' \to \mcB$ which extends~$\sigma_{\mcA}$.\\
(ii) Let $k \in \mbbN$. We say $\mcM$ has the {\bf \em $k$-extension property} 
if it has the $\mcB/\mcA$-extension property whenever $\mcB$ is an (strongly) $l$-coloured
structure, $\mcA$ is a closed substructure of $\mcB$ and $\dim_\mcM(B) \leq k$.
\end{defi}

\noindent
When saying that two $l$-coloured structures $\mcA$ and $\mcA'$ 
{\bf \em agree on $L_{pre}$ and on closed proper substructures}
we mean that $\mcA \uhrc L_{pre} = \mcA' \uhrc L_{pre}$ 
(so in particular, $\cl_{\mcA} = \cl_{\mcA'}$)
and whenever $\mcU$ is a closed substructure of $\mcA$ and $\dim_{\mcA}(U) < \dim_{\mcA}(A)$,
then $\mcA \uhrc U = \mcA' \uhrc U$.

\begin{lma}\label{acceptance of substitutions by l-colourable structures}
Whenever $\mcM$ is (strongly) $l$-coloured, $\mcA$ is a closed substructure of $\mcM$
and $\mcA'$ is an (strongly) $l$-coloured structure which agrees with $\mcA$ on $L_{pre}$ 
and on closed proper substructures,
then there is an (strongly) $l$-coloured structure $\mcN$ such that
$\mcN \uhrc L_{pre} = \mcM \uhrc L_{pre}$,
$\mcN \uhrc A = \mcA'$ and 
if $\mcU$ is a closed substructure of $\mcN$,
$\dim_{\mcN}(U) \leq \dim_{\mcN}(A')$
and $\mcU \neq \mcA'$, then $\mcN \uhrc U = \mcM \uhrc U$.
\end{lma}

\begin{proof}
We only prove the lemma in the case of $l$-coloured structures.
The proof for {\em strongly} $l$-coloured structures is a straightforward modification.
Suppose that $\mcM$ is $l$-coloured, that $\mcA$ is a closed substructure of $\mcM$,
and therefore $l$-coloured. Also assume that $\mcA'$ is $l$-coloured and agrees
with $\mcA$ on $L_{pre}$ and on closed proper substructures.
Observe that by these assumptions and Assumption~\ref{assumptions on languages}, 
for every $X \subseteq A$ we have $\cl_{\mcM}(X) = \cl_{\mcA}(X) = \cl_{\mcA'}(X)$
and $\dim_{\mcM}(X) = \dim_{\mcA}(X) = \dim_{\mcA'}(X)$, so we can omit the subscripts.
The proof splits into three cases.

First suppose that $\dim(A) = 0$.
By parts~(1) and~(2) of the definition of $l$-coloured structure we have $\mcA = \mcA'$
so if $\mcN = \mcM$ then the conclusion of the lemma is satisfied.

Now suppose that $\dim(A) = 1$, so
$\mcA$ is a one dimensional structure and therefore all $a \in A - \cl(\es)$
have the same colour in $\mcA$, say $i$ (that is, $\mcA \models P_i(a)$).
Similarly, $\mcA'$ is a one dimensional structure so all $a \in A' - \cl(\es)$ have the same colour in 
$\mcA'$, say $j$.
Let $\mcN$ be the structure which satisfies the following conditions:
\begin{itemize}
\item[\textbullet] $\mcN \uhrc L_{pre} = \mcM \uhrc L_{pre}$, so in particular $N = M$.

\item[\textbullet] For every $R \in V_{rel}$, $R^{\mcN} = \es$.

\item[\textbullet] For every $a \in M - A$
and every $m \in \{1, \ldots, l\}$, \\
$\mcN \models P_m(a)$ $\Longleftrightarrow$ 
$\mcM \models P_m(a)$.

\item[\textbullet] For every $a \in A - \cl(\es)$, $\mcN \models P_j(a)$.
\end{itemize}
Then $\mcN$ is $l$-coloured, for trivial reasons, and has the required properties which is easily checked.

Finally suppose that $\dim(A) = k+1$ where $k \geq 1$.
Define $\mcN$ as follows:
\begin{itemize}
\item[\textbullet] $\mcN \uhrc k = \mcM \uhrc k$, so in particular $\mcN \uhrc L_{pre} = \mcM \uhrc L_{pre}$.

\item[\textbullet] Whenever $U$ is a closed subset of $M = N$,
$\dim(U) = k+1$ and $U \neq A$, then $\mcN \uhrc U = \mcM \uhrc U$.

\item[\textbullet] $\mcN \uhrc A = \mcA'$.

\item[\textbullet] Whenever $\bar{a} \in M$, $\dim(\bar{a}) > k+1$ and
$R \in V_{rel}$, then $\bar{a} \notin R^{\mcN}$.
\end{itemize}
It remains to prove that $\mcN$ is $l$-coloured.
Since $\mcN \uhrc k = \mcM \uhrc k$, where $k \geq 1$, it follows that $\mcN \uhrc L_{col} = \mcM \uhrc L_{col}$
and hence conditions~(1)--(3) in the definition of $l$-coloured structure are satisfied.

Now we consider condition~(4).
Suppose that $\mcN \models R(\bar{a})$ where $R \in V_{rel}$.
We need to show that there are $b, c \in \cl(\bar{a}) - \cl(\es)$ such
that $b$ and $c$ have different colours.
By the last part of the definition of $\mcN$ we may assume that $\dim(\bar{a}) \leq k+1$.
If $\dim(\bar{a}) \leq k$ then, by the first part of the definition of $\mcN$,
and the assumption that $\mcM$ is $l$-coloured it follows that $b, c \in \cl(\bar{a}) - \cl(\es)$
with different colours exist.
Now suppose that $\dim(\bar{a}) = k+1$.
If $\cl(\bar{a}) \neq A$, then, by the second part of the definition of $\mcN$
and the assumption that $\mcM$ is $l$-coloured,
there are $b, c \in \cl(\bar{a}) - \cl(\es)$
with different colours.
Finally suppose that $\cl(\bar{a}) = A$.
By the third part of the definition of $\mcN$, $\mcN \uhrc A = \mcA'$, so
$\mcA' \models R(\bar{a})$ and since $\mcA'$ is $l$-colored there are
$b, c \in \cl(\bar{a}) - \cl(\es)$ with different colours in $\mcA'$,
and hence (by the third part of the definition of $\mcN$ again) they have different
colours in $\mcN$. 
\end{proof}

\noindent
In the terminology of \cite{Kop12} (Definition~7.20), 
Lemma~\ref{acceptance of substitutions by l-colourable structures} 
says that, for every $k \in \mbbN$, $\mbK$ and $\mb{SK}$ {\em accept $k$-substitutions over $L_{pre}$}.
Therefore, Assumption~\ref{assumptions on languages} and
Theorems~7.31 and~7.32 in~\cite{Kop12} imply the following:

\begin{thm}\label{ccv1}
Suppose that, for every $k \in \mbbN$, $\mbG$ is polynomially $k$-saturated.\\
(i) For every $k \in \mbbN$, 
\begin{align*}
&\lim_{n\rightarrow\infty}
\delta^{\mbK}_n\big(\{\mcM \in \mbK_n : \mcM \text{ has the $k$-extension property}\}\big) 
\ = \ 1 \quad \text{ and}\\
&\lim_{n\rightarrow\infty}
\delta^{\mb{SK}}_n\big(\{\mcM \in \mb{SK}_n : \mcM \text{ has the $k$-extension property}\}\big)
\ = \ 1.
\end{align*}
(ii) For every $L$-sentence $\varphi$,
$\delta^{\mbK}_n\big(\{\mcM \in \mbK_n : \mcM \models \varphi\}\big)$ approaches either 0 or 1, and 
$\delta^{\mb{SK}}_n\big(\{\mcM \in \mb{SK}_n : \mcM \models \varphi\}\big)$
approaches either 0 or 1, as $n$ tends to infinity.
\end{thm}

\noindent
Now we have a zero-one law for
(strongly) $l$-colour{\em ed} structures, with the dimension conditional probability measure. 
Next, we look att (strongly) $l$-colour{\em able} structures, with a probability measure that
is derived from the dimension conditional measure

\begin{defi}\label{definition of probability measure on C-n and S-n}
For each $n$ and all $X\subseteq\C_n$ and $Y \subseteq \mbS_n$ let
\begin{align*}
\delta_n^\mbC(X) \ &= \ \delta_n^\K \big(\{\M\in\K_n :\M\uhrc\Lrel\in X\}\big), \text{ and} \\
\delta_n^\mbS(X) \ &= \ \delta_n^{\mb{SK}} \big(\{\M\in\mb{SK}_n :\M\uhrc\Lrel\in X\}\big).
\end{align*}
\end{defi}

\noindent
Intuitively, for $X \subseteq \mbC_n$, we can think of $\delta_n^\mbC(X)$ as the probability 
that $\mcM \in \mbC_n$ will belong to $X$ if $\mcM$ is generated by the following procedure: 
start with $\mcG_n$ and randomly add colours to the 1-dimensional subspaces of $\mcG_n$, then add $R$-relations for each $R \in V_{rel}$ in such
a way that the colouring conditions (1)--(4) of Definition~\ref{coldef} are respected but apart from this
in a random fashion, and finally, forget about the specific colouring, that is, consider the reduct to $L_{rel}$.
The probability measure $\delta_n^\mbS$ can be interpreted analogously.
The corollary below states tells that a zero-one law holds for (strongly) $l$-colourable structures
when probability measure $\delta_n^\mbC$ ($\delta_n^\mbS$) is used, in other words,
it states the same thing as Theorem~\ref{zero-one law}.

\begin{cor}\label{0-1 law for l-colourable structures}
Suppose that, for every $k \in \mbbN$, $\mbG$ is polynomially $k$-saturated.
For every $L_{rel}$-sentence $\varphi$, 
$\delta^{\mbC}_n\big(\{\mcM \in \mbC_n : \mcM \models \varphi\}\big)$ approaches either 0 or 1, and 
$\delta^{\mbS}_n\big(\{\mcM \in \mbS_n : \mcM \models \varphi\}\big)$
approaches either 0 or 1, as $n$ tends to infinity.
\end{cor}

\noindent
{\em Proof.}
Let $\varphi$ be an $L_{rel}$-sentence, so in particular it is an $L$-sentence.
Then
\begin{align*}
&\delta_n^\mbC\big(\{\mcM \in \mbC_n : \mcM \models \varphi\}\big) \\
= \ &\delta_n^\mbK\big(\{\mcM \in \mbK_n : \mcM \uhrc L_{rel} \models \varphi \}\big) \quad
\text{ (by the definition of $\delta_n^\mbC$)} \\
= \ &\delta_n^\mbK\big(\{\mcM \in \mbK_n : \mcM \models \varphi\}\big) \quad
\text{ (since $\mcM \models \varphi \ \Leftrightarrow \mcM \uhrc L_{rel} \models \varphi$).} 
\end{align*}
Since $\varphi$ is also an $L$-sentence, Theorem~\ref{ccv1} implies
that $\delta_n^\mbC\big(\{\mcM \in \mbC_n : \mcM \models \varphi\}\big)$ approaches either 0 or 1 as $n \to \infty$.
The proof that $\delta_n^\mbS\big(\{\mcM \in \mbS_n : \mcM \models \varphi\}\big)$ approaches either 0 or 1 as 
$n \to \infty$ is exactly the same; just replace $\mbC_n$ by $\mbS_n$ and $\mbK_n$ by $\mb{SK}_n$.
\hfill $\square$
\\

\noindent
However, neither the theorem nor its proof gives information about for which $L_{rel}$-sentences
$\varphi$ we have $\lim_{n\to\infty} \delta_n^\mbC\big(\{\mcM \in \mbC_n : \mcM \models \varphi\}\big) = 1$,
nor do we get information about structural properties of (strongly) $l$-colourable structures.
The remaining sections deal with these issues.
In hindsight it seems silly that the second author of this article did not notice, in \cite{Kop12}, 
this easy way of proving the zero-one law of (strongly) $l$-colourable structures with trivial pregeometry,
when the measures $\delta_n^\mbC$ (or $\delta_n^\mbS$) are used.
But in \cite{Kop12} emphasis was put on extension axioms, which may explain why the above ``short cut'' to
Corollary~\ref{0-1 law for l-colourable structures} in the case when the underlying pregeomeries 
are trivial was not noticed.

It will sometimes be convenient to think of $l$-colourings as functions that 
assign colours to elements, as done in combinatorics, so we introduce the following terminology.

\begin{defi}\label{definition of l-colouring}
Let $\A$ be an $\Lrel$-structure and let $\gamma: A - \cl(\es) \rightarrow \{1,...,l\}$. 
Let $B$ be a closed subset of $A$.
We say that $B$ is $\gamma$-{\bf \em monochromatic}
if for all $a, b \in B - \cl(\es)$, $\gamma(a) = \gamma(b)$. 
If $B$ is not $\gamma$-monochromatic then it is called $\gamma$-{\bf \em multichromatic}.
If $\gamma(a) \neq \gamma(b)$ whenever $a \in B$ and $b \in B$ are independent, then we call
$B$ {\bf \em strongly $\gamma$-multichromatic}.
If there is no risk of confusion we may just say monochromatic, multichromatic or strongly multichromatic. 
We say that $\gamma$ is a {\bf \em (strong) $l$-colouring} of $\A$ if the following conditions hold:
\begin{enumerate}
\item For every $a \in A - \cl(\es)$, $\cl(a)$ is $\gamma$-monochromatic.
\item If $R\in\Vrel$ and $\A\models R(\bar a)$ then $\cl(\bar a)$ is (strongly) $\gamma$-multichromatic.
\end{enumerate}
\end{defi}

\noindent
Observe that an $L_{rel}$-structure $\mcA$ is (strongly) $l$-colourable, according to 
Definition~\ref{coldef}, if and only if there is an (strong) $l$-colouring 
$\gamma : A - \cl(\es) \to \{1, \ldots, l\}$ of $\mcA$.
We will often want to describe the isomorphism type of some particular structure with
a sentence, which motivates the following definition.

\begin{defi}\label{chardef}
Let $\mcA$ be an (strongly) $l$-colourable structure and let $A = \{a_1, \ldots, a_m\}$
By a {\bf \em characteristic formula of $\mcA$}, with respect to the given enumeration of $A$, 
we mean a quantifier-free $L_{rel}$-formula
$\chi_\mcA(x_1, \ldots, x_m)$ such that if $\mcM$ is an $L_{rel}$-structure such
that the formulas $\theta_n$ define a pregeometry $(M, \cl_\mcM)$ 
and $\mcM \models \chi_\mcA(b_1, \ldots, b_m)$, then the map $a_i \mapsto b_i$, for $i = 1, \ldots, m$,
is an embedding of $\mcA$ into $\mcM$.
Similarly we define a characteristic formula of an (strongly) $l$-coloured structure.
Note that such formulas exist because of the definition of (strongly) $l$-colourable (or $l$-coloured) structures
and Assumption~\ref{assumptions on languages}~(7) 
(see also Remark~\ref{remark on assumptions on pregeometry}~(iii)).
\end{defi}

\section{Definability of strong $l$-colourings}~\label{sc}

\noindent 
In this section we study strongly $l$-coloured structures, where $l \geq 2$ (as always). 
If $a$ and $b$ are elements of a strongly $l$-coloured structure and $\mcM \models P_i(a) \wedge P_i(b)$
for some $i \in \{1, \ldots, l\}$, then we say that $a$ and $b$ have the same colour.
The main result of this section, which is essential for the proof of 
Theorem~\ref{definability of strong colourings}, which is finished in Section~\ref{cc},
is the following: there are $k_0 \in \mbbN$ and an $L_{rel}$-sentence $\xi(x,y)$ such that
\begin{itemize}
\item if $\mcM$ is strongly $l$-coloured, $a, b \in M - \cl(\es)$
and $\mcM \models \xi(a,b)$, then $a$ and $b$ have the same colour, and
\item if $\mcM$ is strongly $l$-coloured and has the $k_0$-extension property and
$a, b \in M - \cl(\es)$, then
$\mcM \models \xi(a,b)$ if and only if $a$ and $b$ have the same colour.
\end{itemize}

\noindent
The definition of strongly $l$-colourable structures implies that if $\mcM$ is strongly $l$-colourable,
$R \in V_{rel}$ is an $r$-ary relation symbol (so $r \geq 2$), $\mcM \models R(a_1, \ldots, a_r)$ and
$b, c \in \cl(a_1, \ldots, a_r) - \cl(\es)$ are independent, then $a$ and $b$ must have different colours.
It follows that if $a_1, \ldots, a_r \in M$ and the number of 1-dimensional subspaces (i.e. closed subsets) of $\cl(a_1, \ldots, a_r)$ is larger than $l$, then $\mcM \not\models R(a_1, \ldots, a_r)$.

\begin{exa}\label{example of strongly l-coloured structure with no relations}
Suppose that $F = \mbbZ_2$ is the 2-element field and, for every $n \in \mbbN$, $\mcG_n$ is 
an $n$-dimensional vector space over $F$, as in Example~\ref{example of vector space over finite field}.
Let $l = 2$. For every 2-dimensional subspace $V$ of $\mcG_n$ ($n \geq 2$), the number
of 1-dimensional subspaces of $V$ is $2^2 -  1 = 3 > l$.
So, with these assumptions, if $\mcM$ is strongly 2-coloured then $R^\mcM = \es$ for every $R\in \Vrel$.
But if, instead, $l > 2$ then it is possible that $R^\mcM \neq \es$ 
for every $R \in V_{rel}$.
\end{exa}

\noindent
Since strongly $l$-coloured structures in which $R$ is interpreted as the empty set for
every $R \in V_{rel}$ are not so interesting, the above example motivates the
following definition and assumption. Observe that by Assumption~\ref{assumptions on languages}~(6),
if $n \in \mbbN$ and $\mcG'$ is a closed substructure of $\mcG_n$, then $\mcG' \cong \mcG_m$
for some $m \in \mbbN$.

\begin{defi}\label{definition of T(s) and t}
(i) If $A$ is a closed subset of $G_n$, for some $n$, then let $D(A)$ be the 
number of 1-dimensional subspaces of $A$.\\
(ii) For every $d \in \mbbN$, let $t(d)$ be the maximum of $D(A)$ where $A$ is a subspace
of $\mcG_n$ for some $n$ and $\dim_{\mcG_n}(A) \leq d$.\\
(iii) Let $t = \max\{ d \in \mbbN : t(d) \leq l\}$.
\end{defi}

\noindent
Note that if $\dim_{\mcG_n}(A) > l$ then $D(A) > l$, so $t \leq l$.
In Example~\ref{example of strongly l-coloured structure with no relations}
we have $t(0) = 0, t(1) = 1, t(2) = 3$ and $t(3) = 8$, so if $l = 2$ then $t = 1$.
If, in the same example, $l \in \{3, \ldots, 7\}$, then $t = 2$; if $l = 8$, then $t = 3$, and so on.
In order that the arguments that follow work out {\em we assume that $t \geq 2$}.
This is equivalent with the condition, in Theorem~\ref{definability of strong colourings},
that for every $n \in \mbbN$, every 2-dimensional subspace of $\mcG_n$ has at most~$l$ different
1-dimensional subspaces.

Let the relation symbols of $V_{rel}$ be $R_1,...,R_\tau$ with arities $r_1,...,r_\tau \geq 2$. 
Without loss of generality we assume that $r_1$ is the smallest among these arities. 
By Assumption~\ref{assumptions on languages} there are $L_{pre}$-formulas $\theta_0$ and $\theta_1$
such that if $\mcM$ is strongly $l$-coloured (or strongly $l$-colourable),
then $\mcM \models \theta_0(a)$ $\Longleftrightarrow$ $a \in \cl(\es)$, and
$\mcM \models \theta_1(a,b)$ $\Longleftrightarrow$ $b \in \cl(a)$.
Since $L_{pre} \subseteq L_{rel}$, this justifies the use of notation like `$x \in \cl(y)$' 
when specifying $L_{rel}$-formulas.

The idea of the formula $\xi(x,y)$ defined below is that whenever $a$ and $b$ do not belong to the closure of $\es$
and $\xi(a,b)$ holds, then $a$ and $b$ must have the same colour (and the converse implication
holds if the structure that $a$ and $b$ come from has the $k$-extension property for large enough $k$).
This is achieved by saying that if $a$ and $b$ are independent then there are
$c_2, \ldots c_l$ such that every pair of distinct elements from $\{a, c_2, \ldots, c_l\}$
is independent and belongs to an $R_1$-relationship, thus forcing them to have different
colours. The same is said about pairs of distinct elements from $\{b, c_2, \ldots, c_l\}$,
thus forcing the elements of every such pair to have different colours. As $c_2, \ldots, c_l$
use up $l-1$ colours and there are only $l$ colours, this forces $a$ and $b$ to have the same colour.
In the following definition we will use notation like $\os{l}{\us{i=1}{\exists}}\os{r_1}{\us{j=1}{\exists}}x_{i,j}$ which is the same as saying $\exists x_{1,1} \ldots x_{1,r} \exists x_{2,1} \ldots x_{2,r} \exists x_{l,1} \ldots x_{l,r_1}$. Moreover, we use triples $(\bullet, \bullet, \bullet)$
to index different variables $z_{(\bullet, \bullet, \bullet)}$.

\begin{defi}
Let $\xi(x,y)$ denote the following $\Lrel$-formula: 
\begin{align*}
&x \in \cl(y) \ \vee \ y \in \cl(x) \ \vee \ \exists y_2,...,y_l \os{l}{\us{i=2}{\exists}} \>\> \os{r_1-2}{\us{j=1}{\exists}}z_{(x,i,j)} \os{l}{\us{i=2}{\exists}} \>\> \os{r_1-2}{\us{j=1}{\exists}}z_{(y,i,j)}
\os{l}{\us{k=2}{\exists}}\>\>\os{k-1}{\us{i=2}{\exists}} \>\> \os{r_1-2}{\us{j=1}{\exists}}z_{(k,i,j)} \\
&\Bigg[\bigwedge_{i=2}^{l} \Big(R_1(x,y_i,z_{(x,i,1)},..., z_{(x,i,r_1-2)}) \ \wedge \ y_{i} \notin \cl(x) \ \wedge \  R_1(y,y_i,z_{(y,i,1)},..., z_{(y,i,r_1-2)}) \\ 
&\wedge \ y_{i} \notin \cl(y) \Big) \ \wedge \ 
\bigwedge_{j=2}^{i-1}\Big(R_1(y_i,y_j,z_{(i,j,1)},..., z_{(i,j,r_1-2)} ) \ \wedge \ y_i \notin \cl(y_j)\Big) \Bigg]. 
\end{align*}
\end{defi}

\noindent 
The variables $z_{(k,i,j)},z_{(x,i,j)}$ and $z_{(y,i,j)}$ have the function of ``fillers'' to get the
the right length, $r_1$, of the tuples.
In the case $r_1=2$ they are not needed and $\xi$ will look like this:
\begin{align*} 
x \in  \cl(y) \ &\vee \ y \in \cl(x) \ \vee \\
\exists y_2,...,y_l \Bigg[
&\bigwedge_{i=2}^{l} 
\Big(R_1(x,y_i) \ \wedge \ y_{i} \notin \cl(x) \ \wedge \ R_1(y,y_i) \ \wedge y_{i} \notin \cl(y)\Big)
\ \wedge \\ 
&\bigwedge_{j=2}^{i-1}\Big( R_1(y_i,y_j) \ \wedge \ y_i \notin \cl(y_j)\Big)\Bigg]. 
\end{align*}

\begin{lma}~\label{strl1}
If $\M$ is strongly $l$-coloured, $a,b \in M-\cl(\emptyset)$  and  $\M \models \xi(a,b)$ then $a$ and $b$ have the same colour in $\M$, i.e.  for some $i =1,...,l$ we have  $\M \models P_i(a) \wedge P_i(b)$.
\end{lma}

\begin{proof}
Let $\M$ be strongly $l$-coloured.
We assume that $\M \models \xi(a,b)$ and $a,b\notin \cl(\emptyset)$. If $a\in \cl(b)$ then we obviously are done by the definition of a colouring, hence assume that $a$ and $b$ are independent.
Each $y_i$ that witness the truth of $\xi(a,b)$ must have a different colour from $a$ since they are independent and included in a tuple $(a,y_i,z_{(a,i,1)},...,z_{(a,i,r_1-2)})\in R_1^\mcM$. In the same way each $y_i$ must have different colour from $b$. 
In the same way as for $a$ and $b$, looking at the definition of $\xi$, we get that $y_i$ and $y_{j}$ must have different colour in $\M$ if $i \neq j$. 
Hence we can conclude that all the elements $a, y_{2},..., y_{l}$ have different colours and all the elements $b,y_{2},...,y_{l}$ have different colours. But since $\M$ is coloured by only $l$ different colours this implies, by the pigeon hole principle, that $a$ and $b$ must have the same colour.
\end{proof}

\noindent 
For the rest of this section, let 
\[k_0  = t(l+1)l.\]
We will now prove that if $\mcM$ is strongly $l$-colourable with the $k_0$-extension property
and $a, b \in M - \cl(\es)$ have the same colour in $\mcM$, then
$\mcM \models \xi(a,b)$.
This will be done by defining a structure $\B$ which has the same relations as described by $\xi$, and showing that this structure is strongly $l$-colourable.
Then we show that if $a$ and $b$ have the same colour in a structure with the $k_0$-extension property, then they are included in a copy of $\B$ in such a way that, by construction of $\B$, $\xi(a,b)$ holds.

\begin{lma} ~\label{strl2}
Let $\M$ be strongly $l$-coloured and assume that $\M$ has the $k_0$-extension property.
If $a,b\in M-\cl(\emptyset)$ and $\M\models P_i(a) \wedge P_i(b)$ for some $i\in\{1,...,l\}$,
then $\M \models \xi(a,b)$.
\end{lma}

\begin{proof}
If we would not be interested in being able to easily adapt the follwing argument
to the context where $R_1$ is always interpreted as an irreflexive and symmetric relation,
then some parts of the argument could be simplified 
(see Remark~\ref{remark about irreflexive symmetric relations}).
Without loss of generality we may assume that $\M\models P_1(a)\wedge P_1(b)$ where $a,b \in M - \cl(\es)$. 
If $a\in \cl(b)$ then $\M \models \xi(a,b)$ by definition, hence assume that $a\notin \cl(b)$. 
Let $\A = \M \reduct \cl(a,b)$ and choose elements $v_2,...,v_l\in M$ and elements $u_{(a,i,j)},u_{(b,i,j)},u_{(k,i,j)} \in M$ for each $2\leq i \leq l, 1\leq j \leq t-2$ and 
$2\leq k\leq l-1, k\not =i$ such that the set $S$
containing exactly the elements $a$, $b$, $v_2, \ldots, v_l$ and $u_{(a,i,j)},u_{(b,i,j)},u_{(k,i,j)}$,
for $i,j,k$ as indicated above, is an independent set.
Such a choice of elements from $M$ is possible because we assume that $\mcM$ has 
the $k_0$-extension property where $k_0 = t(l+1)l$.\footnote{
By Assumption~\ref{assumptions on languages} there is $\mcG_n$ with dimension $k_0$
and hence there is a strongly $l$-coloured structure $\mcB$ with dimension $k_0$.
By Assumption~\ref{assumptions on languages} and the definition of strongly $l$-coloured structures
it follows that $\mcB \uhrc \cl_\mcB(\es) \cong \mcM \uhrc \cl_\mcM(\es)$ and since, 
letting $\mcA = \mcB \uhrc \cl_\mcB(\es)$, 
$\mcM$ has the $\mcB/\mcA$-extension property it follows that $\mcM$ contains an isomorphic copy $\mcB'$
of $\mcB$ and $B'$ contains an independent set of cardinality $k_0$.}
Let $\B_0$ be the substructure of $\mcM \uhrc L_{pre}$ with universe
$\cl(S)$, or equivalently, $\mcB_0 = \big(\mcM \uhrc \cl(S)\big) \uhrc L_{pre}$. 
Note that $\A\uhrc L_{pre}\subseteq \B_0$.
Define $\B$ to be the $L$-structure which is created by expanding $\B_0$ to an $L$-structure in the following way.
We know already that $\A\reduct L_{pre}\subseteq \B\reduct L_{pre}$, so for each $i\in\{1,...,\rho\}$, every $R_i\in V-V_{pre}$, and every $\bar{a}\in A^{r_i}$, we let $\bar{a} \in R_i^\B \Leftrightarrow \bar{a} \in R_i^\A$, and for each $j\in\{1,...,l\}$ and $a \in A$ we let $a \in P_j^\B\Leftrightarrow a \in P_j^\A$. 
In this way we obviously get that $\A\subseteq \B$ as $L$-structures, no matter how we define, in $\B$, 
interpretations on tuples whose range are not included in $A$.
For every relation symbol $R_i\in \Vrel - \{R_1\}$ and $\bar c\in B^{r_i}-A^{r_i}$ let $\B\not \models R(\bar c)$. 
For each $i\in \{2,...,l\}$ and $i<j\leq l$ fix arbitrary elements 
\begin{align*}
&w_{(a,i,1)},...,w_{(a,i,r_1-2)} \ \in \ \cl(a,v_i,u_{(a,i,1)},...,u_{(a,i,t-2)}),\\ 
&w_{(b,i,1)},...,w_{(b,i,r_1-2)} \ \in \ \cl(b,v_{i},u_{(b,i,1)},...,u_{(b,i,t-2)}) \quad \text{ and}\\  &w_{(j,i,1)},...,w_{(j,i,r_1-2)} \ \in \ \cl(v_{j},v_{i},u_{(j,i,1)},...,u_{(j,i,t-2)}).
\end{align*}
Define $R_1^\B$ on $B^{r_1} - A^{r_1}$ in such a way that, for each $2\leq i\leq l$,
\[\B \models \ R_1 (a,v_i,w_{(a,i,1)},...,w_{(a,i,r_1-2)}) \ \wedge \ R_1(b,v_i,w_{(b,i,1)},...,w_{(b,i,r_1-2)}) \]
\[\bigwedge_{k=i+1}^l R_1(v_k,v_i,w_{(k,i,1)},...,w_{(k,i,r_1-2)}), \]
and such that $R_1^\B$ holds for no other tuples than those indicated in the argument above.

In order to complete the definition of $\mcB$ as an $L$-structure we need to
define the interpretations $P_1^\mcB, \ldots, P_l^\mcB$ on elements in $B - A$.
When saying that a 1-dimensional subspace (closed subset) $Q$ gets the colour $i$
we mean that for all $a \in Q - \cl(\es)$, $a \in P_i^\mcB$.
Now we define an $l$-colouring, in $\mcB$, on $B-A$ according to the following five steps, where
we recall that $\cl(a)$ and $\cl(b)$ have colour 1 since, by assumption, $\mcM \models P_1(a) \wedge P_1(b)$:
\begin{itemize}
\item[(1)] For $i = 2, \ldots, l$, $\cl(v_i)$ get the colour $i$.
\item[(2)] By the definition of $t$ and the assumption that $t \geq 2$
it is, for every $i = 2, \ldots, l$, possible to colour all 1-dimensional subspaces of
\( \cl(a,v_i,u_{(a,i,1)},...,u_{(a,i,t-2)}) \)
which have not yet been assigned colours
with the colours $1, \ldots, l$ in such a way that~(1) and~(2) hold and any two different
1-dimensional subspaces of this space get different colours.
\item[(3)] As in~(3) it is possible, for every $i = 1, \ldots, l$, to colour all
1-dimensional subspaces of
\( \cl(b,v_{i},u_{(b,i,1)},...,u_{(b,i,t-2)}) \)
which have not yet been assigned colours
with the colours $1, \ldots, l$ in such a way that~(1) and~(2) hold and any two different
1-dimensional subspaces of this space get different colours.
\item[(4)] As in~(3) and~(4) it is possible, for every $i = 1, \ldots, l$, to colour all
1-dimensional subspaces of
\( \cl(v_{j},v_{i},u_{(j,i,1)},...,u_{(j,i,t-2)}) \)
with the colours $1, \ldots, l$ in such a way that~(1) and~(2) hold and any two different 1-dimensional
subspaces of this space get different colours.
\item[(5)] For every 1-dimensional subspace $Q \subseteq B$ that has not yet been assigned a colour, 
give $Q$ the colour $1$.
\end{itemize}

\begin{claim} The $L$-structure $\B$ is a strongly $l$-coloured structure.
\end{claim}
\begin{proof}[Proof of claim.] By the last part of the definition of the colouring of $\mcB$
and since $\mcA \subseteq \mcB$ where $\mcA$ is a substructure of $\mcM$, we know that each element has attained at least one colour, so colouring condition~(1) of Definition~\ref{coldef} is satisfied. 
The second colouring condition is also satisfied because $\mcA \subseteq \mcB$ and $\mcA \subseteq \mcM$.
If we apply Lemma~\ref{pregl1} we get the following, 
for all $i, j , k$ under consideration,
\begin{align*}
\cl(v_k,v_i,u_{(k,i,1)},...,u_{(k,i,t-2)}) \ &\cap \ \cl(b,v_i,u_{(b,i,1)},...,u_{(b,i,t-2)}) \ = \ \cl(v_{i})
\text{ if } k \neq i,\\
\cl(b,v_{i},u_{(b,i,1)},...,u_{(b,i,t-2)}) \ &\cap \ \cl(a,v_i,u_{(a,i,1)},...,u_{(a,i,t-2)}) \ = \ \cl(v_{i}),\\
\cl(a,v_{i},u_{(a,i,1)},...,u_{(a,i,t-2)}) \ &\cap \ \cl(a,v_{j},u_{(a,j,1)},...,u_{(a,j,t-2)}) \ = \ \cl(a)
\text{ if } i \neq j,\\
\cl(b,v_{i},u_{(b,i,1)},...,u_{(b,i,t-2)}) \ &\cap \ \cl(b,v_{j},u_{(b,j,1)},...,u_{(b,j,t-2)}) \ = \ \cl(b)
\text{ if } i \neq j,\\
\cl(a,v_{i},u_{(a,i,1)},...,u_{(a,i,t-2)}) \ &\cap \ \cl(v_{k},v_{i},u_{(k,i,1)},...,u_{(k,i,t-2)}) \ = \ \cl(v_{i})
\text{ if } i \neq k \text{ and}\\
\cl(v_{k},v_{i},u_{(k,i,1)},...,u_{(k,i,t-2)}) \ &\cap \ \cl(v_{j},v_{i},u_{(j,i,1)},...,u_{(j,i,t-2)})
\ = \ \cl(v_{i}) \text{ if } k \neq j.
\end{align*}
This shows that the steps~(1)--(6) did not give more than one colour to any element of $B$,
and, from the construction it is also clear that dependent elements that do not belong to the 
closure of $\es$ have obtained the same colour.
The colouring restricted to $\A\subseteq \B$ does, since $\A \subseteq \M$ and $\mcM$ is an $l$-coloured structure, satisfy all the colouring conditions.
Hence the third colouring condition is satisfied for $\mcB$.
If $\mcB \models R_p(\bar{a})$ for some $R_p \in V_{rel}$, then either $\bar{a} \subset A^{r_p}$
in which case the colouring conditions~(4) and~(5) are satisfied since $\mcA \subseteq \mcM$
is $l$-coloured, or $R_p = R_1$ and $\bar{a}$ is identical to one of the following tuples
\begin{align*}
&(a,v_i,w_{(a,i,1)},...,w_{(a,i,r_1-2)}), \\
&(b,v_i,w_{(b,i,1)},...,w_{(b,i,r_1-2)}), \ \text{ or } \\
&(v_k,v_i,w_{(k,i,1)},...,w_{(k,i,r_1-2)}),
\end{align*}
for some $i, k$.
By the choice of these tuples and the steps~(1)--(6) above, it follows that
whenever $a, b \in \cl(\bar{a}) - \cl(\es)$ and $a$ is independent from $b$,
then $a$ and $b$ have different colours. Hence colour conditions~(4) and~(5)
are satisfied and we have proved that $\mcB$ is strongly $l$-coloured.
\end{proof}

\noindent \textit{Continuing the proof of Lemma \ref{strl2}.} By the claim, $\B$ is a strongly $l$-coloured $L$-structure and, by the definition of $\mcB$, $\A$ is a closed substructure of $\B$.
Since $B = \cl(S)$ we know that $\dim(B) \leq t(l+1)l = k_0$. 
As $\M$ has the $k_0$-extension property and $\mcA$ is a closed substructure of $\mcM$, there are a closed substructure $\B'\subseteq \M$ and an isomorphism $f : \B' \rightarrow \B$ with which extends the identity function on $\A$,
so $\mcA \subseteq \mcB'$. From the definition of $\B$ we get that $\M\models\xi(a,b)$.
\end{proof}

\noindent Using Lemmas \ref{strl1} and \ref{strl2} we directly get the following:

\begin{cor}~\label{strxi}
If $\M$ is a strongly $l$-coloured structure with the $k_0$-extension property and $a,b\in M-\cl(\emptyset)$ then 
\[\M\models\xi(a,b)\hspace{1cm}\Longleftrightarrow\hspace{1cm}\M\models P_i(a) \wedge P_i(b) \text{ for some }i\in\{1,...,l\}. \]
\end{cor}

\begin{rmk}\label{remark about irreflexive symmetric relations}
Suppose that we only consider $L$-structures in which $R_1$ 
(a symbol of $V_{rel}$ with minimal arity) is interpreted as an 
irreflexive and symmetric relation.
Then, for every $i = 2, \ldots, l$, the elements $a,v_i,w_{(a,i,1)},...,w_{(a,i,r_1-2)}$ from the proof of
Lemma~\ref{strl2} must be different from each other, where $w_{(a,i,1)},...,w_{(a,i,r_1-2)} \in \cl_\mcM(a,v_i,u_{(a,i,1)},...,u_{(a,i,t-2)})$,
and similarly for the sequences 
$b,v_i,w_{(b,i,1)},...,w_{(b,i,r_1-2)}$ and $v_k,v_i,w_{(k,i,1)},...,w_{(k,i,r_1-2)}$.
Since $a, b, v_2, \ldots, v_l$ are different, by construction, this can be achieved if, for
every $n$, every closed $t$-dimensional subset of $G_n$ has cardinality at least $r_1$,
where $r_1$ is the arity of $R_1$.
Moreover, in the construction of $\mcB$ we must enlarge $R_1^\mcB$ so that 
whenever $\mcB \models R_1(c_1, \ldots, c_r)$ then $\mcB \models R_1(c_{\pi(c_1)}, \ldots, c_{\pi(c_r)})$
for every permutation $\pi$ of $\{1, \ldots, r\}$. 
These changes do not affect the way in which $\mcB$ is coloured in steps~(1)--(6).
\end{rmk}

\section{Definability of $l$-colourings}~\label{wc}

\noindent
Recall Assumptions~\ref{assumptions on languages}.
In this section we assume throughout that for some finite field $F$ one of the following three cases
hold for every $n \in \mbbN$: (a) $G_n$ is an $n$-dimensional vector space over $F$ and $\cl_{\mcG_n}$
is the linear closure operator, or (b) $G_n$ is an $n$-dimensional affine space over $F$ and
$\cl_{\mcG_n}$ is the affine closure operator, or (c) $G_n$ is an $n$-dimensional projective space
over $F$ and $\cl_{\mcG_n}$ is the projective closure operator.
Moreover, we assume that the language $L_{pre}$ with which $\cl_{\mcG_n}$ is defined, according
to Definition~\ref{spredef}, is either $L_{gen}$ from Example~\ref{generic example of first-order pregeometry}
with the same interpretations of symbols as explained in that example,
or, provided we are in case (a) above, we have $L_{pre} = L_F$ where $L_F$ is like in
Example~\ref{example of vector space over finite field} with the same interpretations of symbols
as explained there.

The assumption about the language $L_{pre}$ guarantees that there is no other structure on $\mcG_n$ than that which is
needed for defining the pregeometry. Therefore the following result, essentially of basic linear algebra,
applies in the present context.

\begin{lma}\label{indendent elements map isomorphically to independent elements}
Let $n, m \in \mbbN$. If $\{a_1, \ldots, a_k\} \subseteq G_n$ and $\{b_1, \ldots, b_k\} \subseteq G_m$
are independent sets, then there is an $L_{pre}$-isomorphism from 
$\cl_{\mcG_n}(a_1, \ldots, a_k)$ to $\cl_{\mcG_m}(b_1, \ldots, b_k)$ which maps $a_i$ to $b_i$
for all $i = 1, \ldots, k$.
\end{lma}

\noindent
In this section we will prove the same kind of result for $l$-colourable structures with underlying pregeometry $\mcG_n$ for some $n$ as we did for strongly $l$-colourable structures in Section~\ref{sc} (where the assumptions
on $\mcG_n$ made here were not needed).
More precisely, we will show that there are $k_0 \in \mbbN$ and an $L_{rel}$-sentence $\xi(x,y)$ such that
\begin{itemize}
\item if $\mcM$ is $l$-coloured, $a, b \in M - \cl(\es)$
and $\mcM \models \xi(a,b)$, then $a$ and $b$ have the same colour, and
\item if $\mcM$ is $l$-coloured and has the $k_0$-extension property and
$a, b \in M - \cl(\es)$, then
$\mcM \models \xi(a,b)$ if and only if $a$ and $b$ have the same colour.
\end{itemize}
We will define a certain $l$-colourable $L_{rel}$-structure $\mcB$
which will be used to define the sought after formula $\xi(x,y)$.
In order to define such $\mcB$ we will use a theorem from structural Ramsey theory about
colourings of vector spaces, projective spaces and affine spaces over a finite field.

\begin{defi}\label{definition of colourings of vector spaces}
Suppose that $(V, \cl)$ is a pregeometry.\\
(i) We call a function $c : V - \cl(\es) \to \{1, \ldots, l\}$ an {\bf \em $l$-colouring} of $(V, \cl)$
if whenever $a, b \in V - \cl(\es)$ and $a \in \cl(b)$, then $c(a) = c(b)$.\\
(ii) Suppose that $c : V - \cl(\es) \to \{1, \ldots, l\}$ is an $l$-colouring of $(V, \cl)$ and
that $W$ is a subspace (i.e. a closed subset) of $V$.
If all $a \in W - \cl(\es)$ are assigned the same colour by $c$, then we call 
$W$ $c$-{\bf \em monochromatic}.
If, in addition, there is no closed $U \subseteq V$ such that $W$ is a proper subset of $U$,
then we call $W$ {\bf \em maximal $c$-monochromatic}.
\end{defi}

\noindent
The following theorem was proved by Graham, Leeb and Rothschild \cite{GLR} and can
also be found (in perhaps more accessible form) in \cite{GRS} (Theorem~9 and Corollary~10
of Section~2.4). Recall that we have fixed a finite vector space $F$.

\begin{thm}~\label{ramseythm} \cite{GLR}
For all  $d,l\in\mathbb{N}$ there is a number $N(d,l)\in\mathbb{N}$ such that if $n\geq N(d,l)$, and
the pregeometry $(V, \cl)$ is isomorphic with an $n$-dimensional vector space, projective space, or affine space over $F$ and $c$ is an $l$-colouring of $(V, \cl)$, then there exists at least one $c$-monochromatic subspace of $(V, \cl)$ with dimension at least $d$.
\end{thm}

\noindent Let $n = N(2,l)$ for $N(d,l)$ in the above theorem, let $\V = \G_n$ and let $c$ be an $l$-colouring of $\V$. 
By our choice of $n$ and Theorem \ref{ramseythm} there exists at least one $c$-monochromatic subspace of $\V$ of dimension at least two and hence there also exists at least one maximal c-monochromatic subspace
of $\V$ of dimension two. 
Let $W^c_1,...,W^c_{t(c)}$ enumerate all the maximal $c$-monochromatic subspaces of $\V$ of dimension at least two, where $t(c)$ depends on the $l$-colouring $c$. 
(This `$t(c)$' has nothing to do with the `$t(d)$' used in the previous section.)
Let $C$ be the set of all $l$-colourings of $\V$.
For each $c \in C$, choose a basis $\{d_1,...,d_{e_c}\} \subseteq \bigcup_{i=1}^{t(c)}W_i$ for the closure of $\bigcup^{t(c)}_{i=1}W_i$,
so in particular, $\bigcup^{t(c)}_{i=1}W_i$ has dimension $e_c$. 
Then let $e=\min \{e_c : c \in C\}$. 
Choose $c_0 \in C$ such that $e_{c_0} = e$ and for every other $l$-colouring $c\in C$ with $e_c = e$ we have that $t(c) \leq t(c_0)$.
For this colouring $c_0$, let $m=t(c_0)$ and let $W_1 = W^{c_0}_1,...,W_m = W^{c_0}_m$.

Assume that the relation symbol $R\in\Vrel$ has minimal arity $r$ among the relation symbols in $\Vrel$, so $r\geq 2$. 
Let $\B$ be the expansion of $\V = \G_n$ to the language $\Lrel$ defined by, for each relation symbol $Q\in\Vrel - \{R\}$, letting $Q^\B = \emptyset$ and defining $R^\B$ in the following way:
\begin{itemize}
\item If $v_1,v_2,...,v_r \in W_i$ for some $i \in \{1,...,m\}$ or if $\cl(v_1,...,v_r) = \cl(v_j)$ 
for some $j \in \{1, \ldots, r\}$, then $\B\models \neg R(v_1,...,v_r)$.

\item If $\{v_1,...,v_r\}\not\subseteq W_i$ for all $i=1,...,m$ and 
$\cl_\mcV \mcV(v_1,...,v_r) \neq \cl(v_j)$ for all $j = 1, \ldots, r$, then $\B\models R(v_1,...,v_r)$.
\end{itemize}
Notice that the second case holds if and only if the first case does not hold, so $\B$ is unambiguously defined. 
Let $b_1, b_2 \in W_1$ be independent (notice that they exist because of the choice of $W_1$) and let $\A = \B\reduct \cl(\{b_1, b_2\})$. Observe that since $A \subseteq W_1$ it follows that for every $Q \in V_{rel}$, $Q^\A = \es$
Let $B = \{b_1, b_2, \ldots, b_\beta\}$ and let $\chi_\mcB(x_1, \ldots, x_\beta)$ be the
characteristic formula of $\mcB$ with respect to the ordering $b_1, b_2, \ldots, b_\beta$ of $B$.
So for every $L_{rel}$-structure $\mcM$ we have $\mcM \models \chi_\mcB(a_1, \ldots, a_\beta)$
if and only if the map $b_i \mapsto a_i$ is an embedding of $\mcB$ into $\mcM$.

\begin{defi}\label{definition of xi-0}
Let $\xi_0(x,y)$ denote the $L_{rel}$-formula  
\[\exists z_3,...,z_\beta \chi_\B(x,y,z_3,...,z_\beta).\]
\end{defi}

\noindent
We will use $\xi_0(x,y)$ to define the formula $\xi(x,y)$ with the properties that we are looking for,
explained in the beginning of this section. Before defining $\xi(x,y)$ we need to assure that
$\xi_0(x,y)$ has certain properties which are given by Lemmas~\ref{bcol}--\ref{wl4}.
Notice that, by construction, $\mcB \uhrc L_{pre} = \mcV$ so $\mcB$ and $\mcV$
have the same universe $B = V$ and $\cl_\mcB$ is the same as $\cl_\mcV$ 
(which is why we skip the subscripts of `$\cl$').

\begin{lma}~\label{bcol}
The function $c_0 : B - \cl(\es) \to \{1, \ldots, l\}$ is an $l$-colouring of $\mcB$
(according to Definition~\ref{definition of l-colouring}). 
Consequently there exists an $l$-coloured structure $\B_0$ such that $\B_0\uhrc\Lrel = \B$ and 
for every $b \in B - \cl(\es)$ and every $i \in \{1, \ldots, l\}$, $\mcB_0 \models P_i(b)$
if and only if $c_0(b) = i$.
\end{lma}

\begin{proof}
We define a $L$-structure $\B_0$ by putting colour on $\B$ through the $l$-colouring $c_0$. 
In other words, we let $\mcB_0$ be the expansion of $\mcB$ to $L$ such that
for every $b\in B - \cl(\es)$, $\B_0\models P_{c_0(b)} \wedge \bigwedge_{j\not = c_0(b)} \neg P_j(b)$.
Then $\B_0\reduct \Lrel \cong \B$ so we just need to prove the following:

\begin{claim} $\B_0$ is $l$-coloured.
\end{claim}

\noindent 
We need to check that~(1)--(4) of Definition \ref{coldef} are satisfied.
Conditions~(1) and~(3) are satisfied since $c_0$ is an $l$-colouring of the underlying
pregeometry $\mcV$ of $\mcB$.
Let $R \in V_{rel}$ be as in the definition of $\mcB$.
Let $Q \in V_{rel}$.
If $Q \neq R$ then, by definition of $\mcB$ and $\mcB_0$, $Q^{\mcB_0} = \es$ so~(2) and~(4) 
are satisfied for such $Q$.
Now we consider the case $Q = R$.
Suppose that $\mcB_0 \models R(a_1, \ldots, a_r)$.
By the definition of $\mcB$ and $\mcB_0$ we have 
\begin{itemize}
\item $\{a_1, \ldots, a_r\} \not\subseteq W_i$ for all $i = 1, \ldots, m$, and
\item $\cl(v_1,...,v_r) \neq \cl(v_j)$ for all $j = 1, \ldots, r$.
\end{itemize}
In particular, $\{a_1, \ldots, a_r\} \not\subseteq \cl(\es)$ so~(2) is satisfied.
As  $\cl(v_1,...,v_r) \neq \cl(v_j)$ for all $j = 1, \ldots, r$,
it follows that $\cl_\mcV(a_1, \ldots, a_r)$ has dimension at least 2. 
If $\cl(a_1, \ldots, a_r)$ would be $c_0$-monochromatic then it would
be included in a maximal $c_0$-monochromatic subspace and, by the first point above, 
this would contradict the assumption (in the construction of $\mcB$) 
that $W_1, \ldots, W_m$ enumerate all maximal $c_0$-monochromatic subspaces of $\mcV$ of dimension at least 2.
Hence $\cl(a_1, \ldots, a_r)$ is not monochromatic,
so~(4) is satisfied. 
Now the claim, and hence the lemma, is proved.
\end{proof}

\noindent The structure $\B_0$ from the previous lemma will be used further on.
Recall the definition of the $L_{rel}$-formula $\xi_0(x,y)$ (Definition~\ref{definition of xi-0}).

\begin{lma}~\label{wl1}
If $\mcM$ is an $l$-coloured structure, $v,w\in M-\cl(\emptyset)$ and $\M \models \xi_0 (v,w)$ then $v$ and $w$ have the same colour, i.e. $\M \models P_i(v) \wedge P_i(w)$ for some $i \in \{1,..., l\}$.
\end{lma}

\begin{proof}
Suppose that $\mcM$ is an $l$-coloured structure.
Observe that if $\mcB'$ is a substructure of $\mcM \uhrc L_{rel}$, then $\mcM$ induces 
a function $c : B' - \cl(\es) \to \{1, \ldots, l\}$
by letting, for every $b' \in B' - \cl(\es)$, $c(b') = i$ if and only if 
$\mcM \models P_i(b')$. 
We call such a function $c$ an $l$-colouring of $\mcB'$ although, strictly speaking,
we can only be sure that it is a colouring of $\mcB'$ (in the sense of 
Definition~\ref{definition of l-colouring}) if $\mcB'$ is a closed substructure of $\mcM$.
Recall the definition of 
$\B$ before Definition~\ref{definition of xi-0} and that $c_0$ is, by Lemma~\ref{bcol},
an $l$-colouring of $\B$. 
The lemma will be proved with the help of the following claim.

\begin{claim}\it Any isomorphism $f$ (if such exists) from the $\Lrel$-structure $\B$ to a substructure $\B' \subseteq \M \uhrc \Lrel$ induces a bijection between the maximal $c_0$-monochromatic subspaces of $\B$ and 
the maximal $c$-monochromatic subspaces of $\B'$, where $c$ is the $l$-colouring of $\B'$ induced by $\M\uhrc \Lrel$.
\end{claim}

\begin{proof}[Proof of the claim.]
Suppose that $f$ is an isomorphism from $\B$ to a substructure $\B'$ of $\mcM \uhrc L_{rel}$.
Assume that $c$ is the $l$-colouring of $\B'$ induced by $\mcM \uhrc L_{rel}$, that is,
for all $b \in B' - \cl(\es)$ and $i \in \{1, \ldots, l\}$, $\M\models P_i(b)$ if and only if $c(b)=i$.
Let $W_1,...,W_m$ enumerate, without repetition, the maximal $c_0$-monochromatic subspaces with dimension
at least 2 of $\V$, and hence of $\B$, 
which where chosen when $\B$ was defined and let $b_1, b_2 \in W_1$ be the two independent elements which 
where chosen in the paragraph before Definition~\ref{definition of xi-0}.
Let $W'_1,...,W'_p$ enumerate, without repetition, the maximal $c$-monochromatic subspaces of $\B' \uhrc L_{pre}$ 
of dimension at least $2$. By Theorem~\ref{ramseythm} this sequence is non-empty.  
We must show that $p = m$ and that there is a permutation $\pi$ of $\{1, \ldots, m\}$
such that $W'_i = f(W_{\pi(i)})$ for all $i = 1, \ldots, m$.

Let $i \in \{1, \ldots, p\}$.
Let $v'_1 \in W'_i$ be arbitrary and, as the dimension of $W'_i$ is at least 2, we can choose 
$v'_2, \ldots, v'_r \in W'_i$ such that $\cl(v'_1, \ldots, v'_r) \neq \cl(v'_j)$ 
for all $j = 1, \ldots, r$. 
We know that $W'_i$ is monochromatic, hence we must have that $\B'\models \neg R(v'_1,...,v'_r)$. 
Choose $v_1, \ldots, v_r \in B$ such that $f(v_j) = v'_j$ for $j = 1, \ldots, r$.
Since $f$ is an isomorphism we have that $\B\models \neg R(v_1,...,v_r)$ and 
$\cl(v_1, \ldots, v_r) \neq \cl(v_j)$ for all $j = 1, \ldots, r$.
By the definition of $\B$, this implies that $v_1,...,v_r \in W_{\pi(i)}$ for some $\pi(i)\in\{1,...,m\}$
(as otherwise we would have $\B \models R(v_1, \ldots, v_r)$, contradicting what we have concluded so far). 

We have already proved that for each $i \in\{1,...,p\}$ there is $\pi(i) \in \{1,...,m\}$ so that 
$W'_i \subseteq f(\W_{\pi(i)})$. 
As $f$ is an isomorphism, and therefore preserves dimension of sets, it follows that 
\[\dim\Big(\bigcup_{i=1}^p W'_i\Big) \leq \dim\Big(\bigcup_{i=1}^m W_i\Big).\]
Observe that the $l$-colouring $c$ of $\B'$ induces an $l$-colouring $c_f$ of $\B$ by letting
$c_f(b) = i$ if and only if $c(f(b)) = i$, for every $b \in B - \cl(\es)$ and every $i \in \{1, \ldots, l\}$.
Therefore, $f^{-1}(W'_1), \ldots, f^{-1}(W'_p)$ is an enumeration of maximal $c_f$-monochromatic
subspaces of $\V$.
It follows that if the above inequality would be strict, then 
$\dim\Big(\bigcup_{i=1}^m W_i\Big)$ would not be minimal among all possible choices of
$l$-colourings of $\V$ and corresponding enumeration of maximal monochromatic subspaces,
and this would contradict the choice of $c_0$.
Hence we conclude that
\[\dim \Big(\bigcup^p_{i=1} W'_i\Big) = \dim \Big( \bigcup^m_{i=1} W_i\Big).\]
Recall that we have showed that for every $i \in \{1, \ldots, p\}$ there is $\pi(i) \in \{1, \ldots, m\}$
such that $W'_i \subseteq f(W_{\pi(i)})$.
Suppose, for a contradiction, that this map $\pi : \{1, \ldots, p\} \to \{1, \ldots, m\}$
is {\em not} surjective.
Then, as $f$ preserves the dimension of sets,
$\bigcup_{i=1}^p f(W_{\pi(i)})$ has strictly smaller dimension than 
$\bigcup_{i=1}^m W_i$. Since $\bigcup_{i=1}^p W'_i \subseteq \bigcup_{i=1}^p f(W_{\pi(i)})$ it
follows that $\bigcup_{i=1}^p W'_i$ has strictly smaller dimension than $\bigcup_{i=1}^m W_i$
which contradicts what we have already proved.
Therefore we conclude that $\pi : \{1, \ldots, p\} \to \{1, \ldots, m\}$ is surjective,
from which it follows that $p \geq m$.

Recall the notation `$t(c)$' used in the definition of $\B$.
By the choice of the colouring $c_0$ of $\V$ (and of $\B$) we have $p=t(c_f)\leq t(c_0)=m$.
As also $p \geq m$ we get $p = m$ and since $\pi$ is surjective (and $p$ finite) it must be bijective. 
\end{proof}

\noindent
Now we continue with the proof of Lemma~\ref{wl1}.
Assume that $\M \models \xi_0(v,w)$. Then there is a $\B' \subseteq \M$ with $B'=\{v,w,b'_3,...,b'_\beta\}$ and $\M\models\chi_\B(v,w,b'_3,...,b_\beta')$. 
Recall the choice of maximal $c_0$-monochromatic subspaces $W_1, \ldots, W_m \subseteq V = B$ and independent 
$b_1, b_2 \in W_1$ in the construction of $\mcB$ before Definition~\ref{definition of xi-0}.
As $\chi_\B$ is the characteristic formula of $\B$ with respect to
an enumeration of $B$ starting with $b_1, b_2, \ldots$, there is an isomorphism $f:\B \rightarrow \B'$ such that $f(b_1) = v$ and $f(b_2)=w$, where $b_1, b_2 \in W_1$. By the claim we have that $f(W_1)$ is a monochromatic subset of $\B'$, with respect to the $l$-colouring $c$ induced by $\mcM$, and since $v,w \in f(W_1)$ it follows that $v$ and $w$ must have the same colour in $\mcM$. 
\end{proof}

\noindent Remember, from before Definition~\ref{definition of xi-0},
that $\A = \B \uhrc \cl(b_1,b_2)$, where $b_1$ and $b_2$ are independent elements of $W_1$. 
Let 
\[k_0 = \max\big(\dim(\mcB), 3\big)\]

\begin{lma}~\label{wl2}
Let $\M$ be an $l$-coloured structure with the $k_0$-extension property,
suppose that $v,w \in M$ and that $\A'$ is a substructure of $\M$ with universe $cl(v,w)$. 
If all elements in $A' - \cl(\es)$ have the same colour and there is an isomorphism $f_0 : \A' \uhrc L_{rel} \rightarrow \A $ such that $f_0(v) = b_1$ and $f_0(w) = b_2$ then $\M \models \xi_0 (v,w)$.
\end{lma}

\begin{proof}
Let $\M$, $v, w \in M - \cl(\es)$ and $\A'$ satisfy the assumptions of the lemma, from which it follows in particular 
that $\A'$ is a closed substructure of $\M$.
Suppose that $f_0 : \A' \uhrc L_{rel} \to \A$ is an isomorphism such that $f_0(v) = b_1$
and $f_0(w) = b_2$. 
Let $\mcB_0$ be the $L$-expansion of $\mcB$ from Lemma~\ref{bcol} and let $\mcA_0 = \mcB_0 \uhrc A$,
so $\mcA_0$ is a closed substructure of $\mcB_0$.
Since, by assumption, all elements of $A' - \cl(\es)$ have the same colour we get $Q^{\mcA'} = \es$
for all $Q \in V_{rel}$.
By the definition of $\mcB_0$, all elements of $A_0 - \cl(\es) = cl(b_1, b_2) - \cl(\es)$ have the
same colour, so $Q^{\mcA_0} = \es$ for all $Q \in V_{rel}$. Let $i$ be the colour of all elements
in $A' - \cl(\es)$. By permuting the colours if necessary we get an $l$-coloured structure $\mcB'_0$
such that if $\mcA'_0 = \mcB'_0 \uhrc A$, then $\mcA'_0 \uhrc L_{rel} = \mcA_0 \uhrc L_{rel}$ and
$f_0$ is an $L$-isomorphism from $\mcA'$ to $\mcA'_0$.
Since $\M$ satisfies the $k_0$-extension property and $\dim(\mcB'_0) = k_0$,
there is an embedding $f:\B'_0 \rightarrow \M$ which extends $f^{-1}_0$. 
Let $\B'$ be the $L_{rel}$-reduct of $\M \uhrc \text{im}(f)$, so $\B\cong\B'$. 
Since $f$ is an $L$-isomorphism (where $L_{rel} \subseteq L$) which extends $f^{-1}_0$ we have that $v,w\in B'$ and
$B'$ can be enumerated in such a way $v,w,b'_3,...,b'_\beta$ that $\M\models \chi_\B (v,w,b'_3,...,b'_\beta)$.
Hence $\M\models\xi_0(v,w)$.
\end{proof}

\begin{lma}~\label{wl3}
Assume that $\M$ is an $l$-coloured structure with the $k_0$-extension property. If $v,w \in \M$ are independent, and have the same colour, then there exists $u \in M - \cl(v,w)$ such that the following holds:
\begin{itemize}
\item[] Let $\A_{v,u} = \M\reduct \cl(v,u)$ and let $\A_{w,u}=\M\reduct \cl(w,u)$. Then $\A_{v,u}$
and $\A_{w,u}$ are monochromatic and there exist isomorphisms $f_{v,u}:A_{v,u}\reduct L_{rel}\rightarrow \A$ and $f_{w,u}: \A_{w,u}\reduct L_{rel} \rightarrow \A$ such that $f_{v,u}(v) = b_1$, $f_{v,u}(u)=b_2$, $f_{w,u}(w)=b_1$ and $f_{w,u}(u)=b_2$. 
\end{itemize}
\end{lma}

\begin{proof}
Suppose that $\mcM$ is $l$-coloured with the $k_0$-extension property 
and assume that $v, w \in M$ are independent from each other
and have the same colour, say 1, without loss of generality.
Let $\mcS$ be any $l$-coloured structure with dimension 3 and let $\mcS' = \cl_\mcS(\es)$.
By the definition of $l$-coloured structures and Assumption~\ref{assumptions on languages},
$\mcS'$ is isomorphic with $\mcM \uhrc \cl_\mcM(\es)$.
Since $k_0 \geq 3$ and $\mcM$ has the $k_0$-extension property it follows that
$\mcM$ has a closed substructure which is isomorphic with $\mcS$.
Therefore $\dim(\mcM) \geq 3$ and hence 
there is $u_0 \in M$ such that $\{v, w, u_0\}$ is an independent set.
Let $C = \cl(v,w)$ and $\mcC = \mcM \uhrc C$.
We will now construct an $l$-coloured structure $\mcD$ and show that $\mcC$ is contained within
an isomorphic copy $\mcD'$ of $\mcD$. The conclusions of the lemma will then
follow easily because of the definition of $\mcD$.

Let $\mcD$ have universe $D = \cl(v,w,u_0)$.
Interpret the symbols in $V_{pre}$ so that $\mcD \uhrc L_{pre}$ is the substructure 
of $\mcM \uhrc L_{pre}$ with universe $D$.
Interpret the symbols of $V_{col} \cup V_{rel}$ so that 
$\mcD \uhrc C = \mcM \uhrc C$.
For all $d \in D - C$ let $\mcD \models P_1(d)$ and $\mcD \not\models P_i(d)$ if $i \neq 1$.
For every $Q \in V_{rel}$, of arity $q$ say, and every $\bar{d} \in D^q - C^q$,
let $\mcD \not\models Q(\bar{d})$.
From the definition it is clear that $\mcD$ is $l$-coloured.

Now we show that if $Q \in V_{rel}$ and $\bar{d} \in \cl(v,u_0)$, then $\mcD \not\models Q(\bar{d})$.
Suppose for a contradiction that $Q \in V_{rel}$, $\bar{d} \in \cl(v,u_0)$ and $\mcD \models Q(\bar{d})$.
By definition of $\mcD$ we have $\bar{d} \in C$ and, by Lemma~\ref{pregl1}, 
$C \cap \cl(v,u_0) = \cl(v)$, so $\bar{d} \in \cl(v) \subseteq C$. By the definition of $\mcD$
we get $\mcM \models Q(\bar{d})$,
which contradicts that $\mcM$ is $l$-coloured (as all members of $\bar{d}$ belong to the same
1-dimensional subspace).
By a similar proof (replace $v$ by $w$) it follows that if
$Q \in V_{rel}$ and $\bar{d} \in \cl(w,u_0)$, then $\mcD \not\models Q(\bar{d})$.

Now let $\mcA_{v, u_0}$ be the substructure of $\mcD$ with universe $A_{v, u_0} = \cl(v, u_0)$.
From the definition of $\mcD$ it follows that all elements of $\mcA_{v, u_0}$ have colour 1.
Recall the definition of the $L_{rel}$-structure $\mcA$ with universe $A = \cl(b_1, b_2)$ before 
Definition~\ref{definition of xi-0}.
As $v$ is independent from $u_0$, it follows from 
Lemma~\ref{indendent elements map isomorphically to independent elements}
that there is an $L_{pre}$-isomorphism $f_{v, u_0} : \mcA_{v, u_0} \uhrc L_{pre} \to \mcA \uhrc L_{pre}$
such that $f_{v, u_0}(v) = b_1$ and $f_{w, u_0}(u_0) = b_2$.
From the definition of $\mcA$ we have $Q^\mcA = \es$ for every $Q \in V_{rel}$.
Since $\mcA_{v, u_0}$ has universe $\cl(v, u_0)$ it follows from what we proved above
and the definition of $\mcA_{v, u_0}$ that $Q^{\mcA_{v, u_0}} = \es$ for every $Q \in V_{rel}$
Therefore, $f$ is also an $L_{rel}$-isomorphism from $\mcA_{v, u_0} \uhrc L_{rel}$
to $\mcA$.
In the same way we can show that if $\mcA_{w, u_0}$ be the substructure of $\mcD$ with universe 
$A_{w, u_0} = \cl(w, u_0)$, then all elements of $\mcA_{w, u_0}$ have colour 1 and there is an isomorphism 
$f_{w, u_0} : \mcA_{w, u_0} \uhrc L_{rel} \to \mcA$ such that
$f_{w, u_0}(w) = b_1$ and $f_{w, u_0}(u_0) = b_2$.

Because of what has been proved above it now suffices to show that
there are a substructure $\mcD' \subseteq \mcM$ such that $\mcC \subseteq \mcD'$
and an isomorphism $f : \mcD \to \mcD'$ such that $f$ is the identity on $C$.
Then $u = f(u_0)$ has the desired property.
Since $\dim(\mcD) = 3 \leq k_0$ and $\mcM$ has the $k_0$-extension property it follows
that, in particular, $\mcM$ has the $\mcD/\mcC$-extension property.
Therefore such $\mcD'$ and $f$ exist.
\end{proof}

\noindent
Now we put together the previous two lemmas to get the following. 

\begin{lma}~\label{wl4}
Assume that $\M$ is $l$-coloured with the $k_0$-extension property. If $v,w \in \M$ are independent and have the same colour then there exists $u\in M- \cl(v,w)$ such that $\M \models \xi_0 (v,u) \wedge \xi_0 (w,u)$.
\end{lma}

\begin{proof}
By Lemma \ref{wl3}, there is $u \in M- \cl(v,w)$ and monochromatic structures $\A_{v,u}, \A_{w,u}\subseteq \M$ with $A_{v,u}=\cl(v,u)$ and $A_{w,u}=\cl(w,u)$, isomorphisms $f_{v,u}:\A_{v,u}\reduct \Lrel \rightarrow \A$ and $f_{w,u}:\A_{w,u}\reduct \Lrel \rightarrow \A$  with $f_{v,u}(v)=f_{w,u}(w)=a$ and $f_{v,u}(u)=f_{w,u}(u) = b$. So by Lemma \ref{wl2} and $f_{v,u}$ we get $\M\models\xi_0(v,u)$ and then, using $f_{w,u}$ and
by Lemma~\ref{wl2}, we get $\M\models\xi_0(w,u)$.
Hence $\M\models\xi_0(v,u)\wedge\xi_0(w,u)$.
\end{proof}

\noindent We can finally define the desired $L_{rel}$-formula $\xi(x,y)$ and prove,
in Corollary~\ref{wcor}, that it has the property of telling whether elements have the same colour or not.

\begin{defi}
Let $\xi(x,y)$ be the $L_{rel}$-formula
\[x \in \cl(y) \ \vee \ \exists z (\xi_0 (x,z) \wedge \xi_0(y,z)).\]
\end{defi}

\noindent
Observe that since $\xi_0(x,y)$ is an existential formula, that is, $\xi_0(x,y)$ 
has the form $\exists \bar{z} \psi(x,y,\bar{z})$ where $\psi$ is quantifier free,
it follows, from the assumptions in the beginning of this section,
that $\xi(x,y)$ is logically equivalent to an existential formula.
This will be used in Section~\ref{cc}.

\begin{cor}~\label{xi(x,y) implies x and y have the same colour}
If $\M$ is $l$-coloured, $v,w\in M-\cl(\emptyset)$, $v \notin \cl(w)$ and $\M \models \xi(v,w)$, 
then $v$ and $w$ have the same colour, i.e. $\M \models P_i(v) \wedge P_i(w)$ for some $i \in \{1,..., l\}$.
\end{cor}

\begin{proof}
If $\M$ is $l$-coloured, $v,w\in M-\cl(\emptyset)$, $v \notin \cl(w)$ and $\M \models \xi(v,w)$, 
then $\M \models \xi_0(v,u) \wedge \xi_0(w,u)$ for some $u \in M$.
By Lemma~\ref{wl1}, $v$ has the same colour as $u$ and $u$ has the same colour as $w$.
Hence, $v$ and $w$ have the same colour.
\end{proof}

\begin{cor}~\label{wcor}
Let $\M$ be $l$-coloured with the $k_0$-extension property.
If $v,w\in M - \cl(\emptyset)$ then
\[ \M \models \xi(v,w) \hspace{10pt} \Longleftrightarrow \hspace{10pt} \text{v and w have the same colour}.\]
\end{cor}

\begin{proof}
Suppose that $v, w \in M - \cl(\es)$ have the same colour.
If $v$ and $w$ are dependent then $v \in \cl(w)$ so $\M \models \xi(v,w)$.
Assume that $v\notin \cl(w)$. By Lemma \ref{wl4}, there is $u \in M$ such that $\M\models\xi_0(v,u)\wedge\xi_0(w,u)$, so by the definition of $\xi$ we get that $\M \models\xi(v,w)$. \\
The opposite direction is proved like Corollary~\ref{xi(x,y) implies x and y have the same colour}
\end{proof}

\section{Almost sure properties and an axiomatisation of the limit theory}~\label{cc}

\noindent
In this section we show that if an $L_{rel}$-formula $\xi(x,y)$ exists which defines the $l$-colouring
of an (strongly) $l$-coloured structure in the sense of~(1) and~(2) of Theorem~\ref{main theorem, general form}
below, then we can draw some conclusions about the asymptotic structure of (strongly) $l$-colurable 
structures and, if $\xi(x,y)$ is existential then we get an explicit axiomatisation of the set of 
sentences with limit probability 1. Theorem~\ref{main theorem, general form} together with
the results in Sections~\ref{preliminaries}--\ref{wc} imply the main results stated in Section~\ref{introduction}.

We recall the notation from Definitions~\ref{coldef},~\ref{definition of probability measures}
and~\ref{definition of probability measure on C-n and S-n}.
So in particular, $\mbK_n$ denotes the set of $l$-coloured structures $\mcM$ such that $\mcM \uhrc L_{pre} = \mcG_n$ and $\delta_n^\mbK$ denotes the dimension conditional measure on $\mbK_n$.
$\mbC_n$ denotes the set of $l$-colourable structures $\mcM$ such that $\mcM \uhrc L_{pre} = \mcG_n$
and $\delta_n^\mbC$ is the probability measure on $\mbC_n$ derived from $\delta_n^\mbK$.
Similarly, $\mb{SK}_n$ denotes the set of strongly $l$-coloured structures $\mcM$ such that
$\mcM \uhrc L_{pre} = \mcG_n$ and $\delta_n^{\mb{SK}}$ denotes the dimension conditional measure 
on $\mb{SK}_n$. $\mbS_n$ denotes the set of strongly $l$-colourable structures $\mcM$
such that $\mcM \uhrc L_{pre} = \mcG_n$ and $\delta_n^\mbS$ is the probability measure on $\mbS_n$
derived from $\delta_n^{\mb{SK}_n}$.
For any $L_{rel}$-sentence $\varphi$, let
\[\delta_n^\mbC(\varphi) \ = \ \delta_n^\mbC\big( \{\mcM \in \mbC_n : \mcM \models \varphi\} \big),\]
and similarly for $\delta_n^\mbS(\varphi)$.
In this section we will prove the following result.

\begin{thm}\label{main theorem, general form}
Suppose that the conditions of Assumption~\ref{assumptions on languages} hold,
that $\mbG = \{\mcG_n : n \in \mbbN\}$ is polynomially $k$-saturated for every $k \in \mbbN$ and that
there exists an $L_{rel}$-formula $\xi(x,y)$ and natural number $k_0$ with the following properties.
\begin{itemize}
\item[(1)] If $\mcM$ is an $l$-coloured structure, $a, b \in M - \cl_\mcM(\es)$
and $\mcM \models \xi(a,b)$,
then $a$ and $b$ have the same colour (i.e. $\mcM \models P_i(a) \wedge P_i(b)$ for some $i \in \{1, \ldots, l\}$).
\item[(2)] If $\mcM$ is an $l$-coloured structure
that has the $k_0$-extension property and $a, b \in M - \cl_\mcM(\es)$,
then $\mcM \models \xi(a,b)$ if and only if $a$ and $b$ have the same colour.
\end{itemize}
Then the following hold:
\begin{itemize}
\item[(i)] The $\delta_n^\mbC$-probability that the following holds for $\mcM \in \mbC_n$
approaches 1 as $n \to \infty$:
\begin{itemize}
\item[] For all $a,b \in M - \cl_\mcM(\es)$, $\mcM \models \xi(a,b)$ if and only if
{\rm every} $l$-colouring of $\mcM$ gives $a$ and $b$ the same colour.
\end{itemize}

\item[(ii)] $\lim_{n\to\infty} 
\delta_n^\mbC\big(\{\mcM \in \mbC_n : \text{ $\mcM$ has a unique $l$-colouring}\}\big) = 1$.

\item[(iii)] $\lim_{n\to\infty} 
\delta_n^\mbC\big(\{\mcM \in \mbC_n : \text{ $\mcM$ is not $l'$-colourable if $l' < l$}\}\big) = 1$.

\item[(iv)] Suppose, in addition, that $\xi(x,y)$ is an existential formula.
Then the set of $L_{rel}$-sentences $\varphi$ such that
$\lim_{n\to\infty} \delta_n^\mbC(\varphi) = 1$ forms a countably categorical theory which can
be given an explicit axiomatization where every axiom is logically equivalent to a sentence
of the form $\forall \bar{x} \exists \bar{y} \psi(\bar{x}, \bar{y})$ where $\psi$ is quantifier-free. 
\end{itemize}
If the assumptions hold for {\rm strongly} $l$-coloured structures, then~(i)--(iv)
hold if every occurence of $\mbC$ is replaced by $\mbS$.
\end{thm}

\noindent
Observe that in Sections~\ref{sc} and~\ref{wc} we have proved, under the
assumptions of Theorems~\ref{definability of strong colourings} and~\ref{definability of colourings},
respectively, that there are a number $k_0$ and an $L_{rel}$-formula $\xi(x,y)$ such that~(1)
and~(2) of Theorem~\ref{main theorem, general form} are satisfied.
Moreover, if all $L_{pre}$-formulas $\theta_n$ which define the pregeometry 
are quantifier free, then the formula $\xi(x,y)$ obtained in Section~\ref{sc} and in Section~\ref{wc}
is logically equivalent to an existential formula.
Therefore Theorems~\ref{definability of strong colourings} and~\ref{definability of colourings}
follow from Theorem~\ref{main theorem, general form} and the results in
Sections~\ref{sc} and~\ref{wc}. 
So it remains to prove Theorem~\ref{main theorem, general form}.

\subsection*{Proof of Theorem~\ref{main theorem, general form}}

\noindent
The proof is exactly the same in the case of $l$-colourable structures as in the
case of strongly $l$-colourable structures.
Therefore we will speak only of `$l$-colourable (or coloured) structures' and use
the notations $\mbK_n$, $\mbC_n$, $\delta_n^\mbK$ and $\delta_n^\mbC$.
(If we replace the mentioned terminology and notation with 
`strongly $l$-colourable (or coloured) structures', $\mb{SK}_n$, $\mbS_n$, $\delta_n^{\mb{SK}}$
and $\delta_n^\mbS$, then we have a proof for strongly $l$-colourable structures.)
The general idea of the proof is to first define an $L_{rel}$-theory $T_\mbC$ such that for every 
$\varphi \in T_\mbC$, $\lim_{n\to\infty} \delta_n^\mbC(\varphi) = 1$.
Then it will follow from compactness that $T_\mbC$ is consistent.
The next step is to prove that $T_\mbC$ is complete, which will be done by proving that
it is countably categorical and applying Vaught's theorem. When these steps have been carried out it follows easily,
since (by compactness) $T_\mbC \models \varphi$ implies $\Delta \models \varphi$ for some finite 
$\Delta \subseteq T_\mbC$,
that for every $L_{rel}$-sentence $\varphi$, either $\lim_{n\to\infty} \delta_n^\mbC(\varphi) = 0$
or $\lim_{n\to\infty} \delta_n^\mbC(\varphi) = 1$.

We assume that the conditions of Assumption~\ref{assumptions on languages} hold
and that $\mbG = \{\mcG_n : n \in \mbbN\}$ is polynomially $k$-saturated for every $k \in \mbbN$.
Let $k_0$ be a natural number and $\xi(x,y)$ an
$L_{rel}$-formula such that 
\begin{itemize}
\item[(1)] If $\mcM$ is an $l$-coloured structure, $a, b \in M - \cl(\es)$ and $\mcM \models \xi(a,b)$,
then $a$ and $b$ have the same colour (i.e. $\mcM \models P_i(a) \wedge P_i(b)$ for some $i \in \{1, \ldots, l\}$).
\item[(2)] If $\mcM$ is an $l$-coloured structure
that has the $k_0$-extension property and $a, b \in M - \cl(\es)$,
then $\mcM \models \xi(a,b)$ if and only if $a$ and $b$ have the same colour.
\end{itemize}

\noindent
Without loss of generality we may assume that $k_0 \geq 1$.

\begin{lma}\label{xi-hat is an equivalence relation}
Suppose that $\mcM$ is an $l$-coloured structure that has the $k_0$-extension property.
Then the following hold:
\begin{itemize}
\item[(i)] For every $i \in \{1, \ldots, l\}$ there is $a \in M$ with colour $i$ (i.e. $\mcM \models P_i(a)$). 
\item[(ii)] The formula $\xi(x,y)$ defines an equivalence relation on $M - \cl(\es)$ such that  
for all $a, b \in M - \cl(\es)$ we have
$\mcM \models \xi(a,b)$ if and only if $a$ and $b$ have the same colour.
\item[(iii)] The set $M - \cl(\es)$ is partitioned into exactly $l$ (nonempty) equivalence classes by 
the equivalence relation defined by $\xi(x,y)$.
\end{itemize}
\end{lma}

\noindent
{\em Proof.}
Suppose that $\mcM$ is an $l$-coloured structure that has the $k_0$-extension property.

(i) For every $l$-coloured $\mcN$ and 1-dimensional closed substructure $\mcA \subseteq \mcN$,
all $a \in A - \cl(\es)$ have the same colour, say $j$.
If we change the colour of all $a \in A - \cl(\es)$ to $i$, say, then the resulting
structures is still $l$-coloured. As we assume that $\mcM$ has the $k_0$-extension property
(and $\dim(\mcA) = 1 \leq k_0$) it follows 
(since $\mcA \uhrc \cl_\mcA(\es) \cong \mcM \uhrc \cl_\mcM(\es)$) that $\mcM$ has a substructure that is 
isomorphic with $\mcA$ and therefore some element of $\mcM$ has colour $i$ (where
$i$ is an arbitrary colour).

(ii) Follows directly from~(2).

(iii) By part~(i), for every $i \in \{1, \ldots, l\}$, there is some $a \in M - \cl(\es)$
with colour $i$. Hence, it follows from part~(ii) that $M - \cl(\es)$ is partitioned into 
exactly $l$ different equivalence classes by the relation defined by $\xi(x,y)$.
\hfill $\square$
\\

\noindent
The next lemma proves part~(i) of Theorem~\ref{main theorem, general form}.

\begin{lma}\label{part (i) of the theorem}
With $\delta_n^\mbC$-probability approaching 1 as $n \to \infty$ a structure $\mcM \in \mbC_n$
has the following property:
\begin{itemize}
\item[] For all $a, b \in M - \cl(\es)$ we have
$\mcM \models \xi(a,b)$ if and only if every $l$-colouring of $\mcM$
gives $a$ and $b$ the same colour. (In other words, whenever $\mcM' \in \mbK_n$
and $\mcM' \uhrc L_{rel} = \mcM$, then for all $a, b \in M - \cl(\es)$,
$\mcM \models \xi(a,b)$ $\Longleftrightarrow$ $\mcM' \models P_i(a) \wedge P_i(b)$
for some $i$.)
\end{itemize}
\end{lma}

\noindent
{\em Proof.}
For every $n \in \mbbN$, let
\begin{align*}
\mbX_n^\mbK \ &= \ \big\{\mcM \in \mbK_n : \mcM \text{ has the $k_0$-extension property}\big\}, \text{ and}\\
\mbX_n^\mbC \ &= \ \big\{\mcM \in \mbC_n : \mcM = \mcN \uhrc L_{rel} \text{ for some } \mcN \in \mbX_n^\mbK\big\}.
\end{align*}
By the definition of $\delta_n^\mbC$ we have
\begin{align*}
\delta_n^\mbC\big(\mbX_n^\mbC) \ &= \ 
\delta_n^\mbK\big(\mcM \in \mbK_n : \mcM \uhrc L_{rel} \in \mbX_n^\mbC \}\big) \\
&= \ \delta_n^\mbK\big(\{\mcM \in \mbK_n : \mcM \uhrc L_{rel} = \mcN \uhrc L_{rel} 
\text{ for some } \mcN \in \mbX_n^\mbK\}\big) \ \geq \ 
\delta_n^\mbK\big(\mbX_n^\mbK\big).
\end{align*}
By Theorem~\ref{ccv1} we have $\lim_{n\to\infty}\delta_n^\mbK\big(\mbX_n^\mbK\big) = 1$
and hence $\lim_{n\to\infty}\delta_n^\mbC\big(\mbX_n^\mbC\big) = 1$.
Therefore it suffices to prove that if $\mcM \in \mbX_n^\mbC$ then
\begin{itemize}
\item[($*$)] for every $\mcM' \in \mbK_n$ such that $\mcM' \uhrc L_{rel} = \mcM$ and
all $a, b \in M - \cl(\es)$, we have $\mcM \models \xi(a,b)$ if and only
if $\mcM' \models P_i(a) \wedge P_i(b)$ for some $i \in \{1, \ldots, l\}$.
\end{itemize}
So suppose that $\mcM \in \mbX_n^\mbC$, $\mcM' \in \mbK_n$ and $\mcM' \uhrc L_{rel} = \mcM$.
As $\mcM \in \mbX_n^\mbC$ there is $\mcN \in \mbX_n^\mbK$ such that $\mcN \uhrc L_{rel} = \mcM$.
By definition of $\mbX_n^\mbK$, $\mcN$ has the $k_0$-extension property, 
so by Lemma~\ref{xi-hat is an equivalence relation},
$\xi(x,y)$ defines, in $\mcN$, an equivalence relation on $N - \cl(\es)$ with exactly $l$ equivalence classes.
Since $\xi(x,y) \in L_{rel}$ and $\mcN \uhrc L_{rel} = \mcM$ it follows that
$\xi(x,y)$ defines, in $\mcM$, an equivalence relation on
$M - \cl(\es)$ with exactly $l$ equivalence classes.

Note that if $a, b \in M - \cl(\es)$ and $\mcM \models \xi(a,b)$,
then, by~(1), we have $\mcM' \models P_i(a) \wedge P_i(b)$ for some $i$.
It follows that the equivalence relation defined by $\xi(x,y)$
on $M - \cl(\es)$ refines the equivalence relation induced on $M - \cl(\es)$
by the colouring of $\mcM'$. Since both equivalence relations 
have exactly $l$ equivalence classes it follows that they are the same relation.
In other words, for all $a, b \in M - \cl(\es)$,
$\mcM \models \xi(a,b)$ if and only if $\mcM' \models P_i(a) \wedge P_i(b)$
for some $i$. Hence we have proved~($*$) and the proof of the lemma is finished.
\hfill $\square$
\\

\noindent
Observe that Lemma~\ref{part (i) of the theorem} immediately implies the following which proves
part~(ii) of Theorem~\ref{main theorem, general form}:

\begin{cor}\label{unique colouring}
$\delta_n^\mbC\big(\{\mcM \in \mbC_n : \text{ $\mcM$ has a unique $l$-colouring}\}\big) = 1$
\end{cor}

\noindent
The next corollary proves part~(iii) of
Theorem~\ref{main theorem, general form}. 
It implies that there exists an $l$-colourable structure which cannot be $l'$-coloured if $l' < l$.
This may seem obvious, but if the reader tries to explicitly construct such a structure it may
become apparent that it is, on the level of generality considered here, not a trivial problem. 

\begin{cor}\label{almost surely all colours are needed}
If $1 \leq l' < l$, then 
\[ \lim_{n\to\infty} 
\delta_n^\mbC\big(\{\mcM \in \mbC_n : \mcM \text{ is not $l'$-colourable}\}\big)
\ = \ 1. \]
\end{cor}

\begin{proof}
Let $1 \leq l' < l$.
Note that every $l'$-colouring of a structure (using only the colours $1, \ldots, l'$)
is also an $l$-colouring.
Suppose that $\mcM \in \mbC_n$ has an $l'$-colouring, that is, there is $\mcM' \in \mbK_n$
such that $\mcM' \uhrc L_{rel} = \mcM$ and $(P_i)^{\mcM'} = \es$ for all $i = (l'+1), \ldots, l$. 
If $n$ is sufficiently large then
$\mcM$ also has an $l$-colouring in which all colours $1, \ldots, l$ are used, that is, there
is $\mcM'' \in \mbK_n$ such that $\mcM'' \uhrc L_{rel} = \mcM$ and $(P_i)^{\mcM''} \neq \es$ for 
all $i = 1, \ldots, l$.
Clearly the two colourings of $\mcM$ are not permutations of each other, that is,
there is no permutation $\pi$ of $\{1, \ldots, l\}$
such that for every $i \in \{1, \ldots, l\}$ and every $a \in M - \cl(\es)$
we have $\mcM' \models P_i(a)$ if and only if $\mcM'' \models P_{\pi(i)}(a)$.
Hence, for large enough $n$, if $\mcM \in \mbC_n$ has a unique $l$-colouring and $l' < l$,
then $\mcM$ is not $l'$-colourable.
Therefore Corollary~\ref{almost surely all colours are needed} follows from
Corollary~\ref{unique colouring}.
\end{proof}

\noindent
Now it remains to prove part~(iv) of Theorem~\ref{main theorem, general form}.
{\em So for the rest of this section we add the assumption that $\xi(x,y)$ is an
existential formula.}
We will give an explicit axiomatisation of the set of $L_{rel}$-sentences with asymptotic probability 1
and show that the given axioms form a countably categorical theory.
The axiomatisation of the limit theory 
\[\Big\{\varphi \in L_{rel} : \lim_{n\to\infty} \delta_n^\mbC(\varphi) = 1\Big\} \]
will be denoted $T_\mbC$ and will consist of four disjoint parts, denoted
$T_\xi, T_{pre}, T_{iso}$ and $T_{ext}$. Note that since we know, 
by Corollary~\ref{0-1 law for l-colourable structures}, that 
$\mbC_n$ satisfies a zero-one law when the measure $\delta_n^\mbC$ is used, it follows
that the limit theory is consistent (by compactness) and complete.
We will show that whenever $\varphi \in T_\mbC$, then $\lim_{n\to\infty} \delta_n^\mbC(\varphi) = 1$
and that $T_\mbC$ is countably categorical, hence complete.
It will then follow that, for every $L_{rel}$-sentence $\varphi$, 
$T_\mbC \models \varphi$ if and only if $\lim_{n\to\infty}\delta_n^\mbC(\varphi) = 1$.
The part of the axiomatisation $T_\mbC$ which we denote $T_\xi$ consists of only
one sentence $\varphi_1 \wedge \varphi_2$, where $\varphi_1$ and $\varphi_2$
are defined below.

Recall that, by Assumption~\ref{assumptions on languages}, 
in every $l$-coloured, or $l$-colourable, structure, 
the property ``$x$ belongs to the closure of $\es$'' is defined by the formula $\theta_0(x)$ 
and the property ``$y$ belongs to the closure of $\{x\}$''
is defined by the formula $\theta_1(x,y)$.

\begin{defi}\label{definition of U}{\rm
(i) Let $\mcU$ be an $l$-colourable structure which is not $l'$-colourable if $l' < l$,
and let $p = |U|$. Such $\mcU$ exists by Corollary~\ref{almost surely all colours are needed}\\
(ii) Let $\varphi_1$ be an $L_{rel}$-sentence which expresses that $\xi(x,y)$ defines an 
equivalence relation on the set of elements {\em not} satisfying $\theta_0(x)$.\\
(iii) Let $\varphi_2$ be the following $L_{rel}$-sentence:
\begin{align*}
\exists x_1, \ldots, x_p \Bigg( &\chi_\mcU(x_1, \ldots, x_p) \ \wedge \\
&\bigvee_{\underset{|I| = l}{I \subseteq \{1, \ldots, p \}}} 
\Bigg[ \bigwedge_{i \in I} \neg\theta_0(x_i) \ \wedge \ 
\bigwedge_{\underset{i \neq j}{i,j \in I}} \neg \xi(x_i, x_j) \ \wedge \ 
\forall y \bigg(\theta_0(y) \ \vee \ \bigvee_{i \in I} \xi(y, x_i)\bigg) \Bigg] \Bigg),
\end{align*}
where $\chi_\mcU(x_1, \ldots, x_p)$ is the characteristic formula of $\mcU$ for some
enumeration of $U$.
}\end{defi}

\begin{lma}\label{the property of U}
$\lim_{n\rightarrow \infty} \delta_n^{\mbC} (\varphi_1 \wedge \varphi_2) \ = \ 1$.
\end{lma}

\begin{proof}
Let $\mcU^+$ be an $l$-coloured structure such that $\mcU^+ \uhrc L_{rel} = \mcU$.
By Assumption~\ref{assumptions on languages}~(6) and the definition of $l$-coloured structures there is
a unique, up to isomorphism, $l$-coloured structure of dimension 0.
So if $\mcV = \cl_{\mcU^+}(\es)$ then every $l$-coloured structure has a substructure which 
is isomorphic to $\mcV$. It follows that if $\mcM$ is an $l$-coloured structure which has the
$\mcU^+/\mcV$-extension property, then $\mcM$ has a substructure which is isomorphic to $\mcU^+$ and
therefore $\mcM \uhrc L_{rel}$ has a substructure which is isomorphic to $\mcU$.
Let $k = \max(k_0, \dim(\mcU^+))$.
Note that, by~(1),~(2) and Lemma~\ref{xi-hat is an equivalence relation}, 
every $l$-coloured structure $\mcM$ with
the $k$-extension property has the following properties:
\begin{itemize}
\item $\mcM$ has a substructure which is isomorphic with $\mcU^+$.
\item $\xi(x,y)$ defines an equivalence relation on $M - \cl(\es)$ such that, 
for all $a, b \in M - \cl(\es)$, 
$\mcM \models \xi(a,b)$ if and only if $a$ and $b$ have the same colour.
\end{itemize}
We will now prove that if $\mcM$ is $l$-coloured and has the $k$-extension property,
then $\mcM \models \varphi_1 \wedge \varphi_2$. 

So suppose that $\mcM$ is $l$-coloured and has the $k$-extension property.
Then, as mentioned above, $\mcM$ has a substructure which is isomorphic to $\mcU^+$ and $\xi$ defines an
equivalence relation on $M - \cl(\es)$, so $\mcM \models \varphi_1$.
It remains to show that $\mcM \models \varphi_2$.
For notational simplicity we assume $\mcU^+ \subseteq \mcM$.
Let $U = \{a_1, \ldots, a_p\}$ be an enumeration of $U$ such that $\mcM \models \chi_\mcU(a_1, \ldots, a_p)$.
As at least $l$ different colours are needed to colour $\mcU$, there are 
$a_{i_1}, \ldots, a_{i_l} \in U - \cl(\es)$
such that if $j \neq j'$ then $a_{i_j}$ has a different
colour than $a_{i_{j'}}$, so $\mcM \models \neg\xi(a_{i_j}, a_{i_{j'}})$. 
Let $I = \{i_1, \ldots, i_l\}$.
Since there are only $l$ colours and all $a_{i_1}, \ldots, a_{i_l}$ have different colours,
it follows that every $b \in M - \cl(\es)$ must have the same colour as some
$a_{i_j}$ which implies $\mcM \models \xi(b, a_{i_j})$.
Hence $\mcM \models \varphi_2$.

We have proved that if $\mcM$ is $l$-coloured and has the $k$-extension property,
then $\mcM \models \varphi_1 \wedge \varphi_2$.
Consequently,
\[ \delta_n^\mbK\big(\{\mcM \in \mbK_n : \mcM \text{ has the $k$-extension property}\}\big)
\ \leq \ \delta_n^\mbK(\varphi_1 \wedge \varphi_2),\]
so by Theorem~\ref{ccv1}, $\lim_{n\to\infty}\delta_n^\mbK(\varphi_1 \wedge \varphi_2) = 1$.
Since $\varphi_1 \wedge \varphi_2$ is an $L_{rel}$-sentence we have
$\mcM \models \varphi_1 \wedge \varphi_2$ if and only if $\mcM \uhrc L_{rel} \models \varphi_1 \wedge \varphi_2$,
for every $l$-coloured structure $\mcM$. By the definition of $\delta_n^\mbC$ we get
$\delta_n^\mbK(\varphi_1 \wedge \varphi_2) = \delta_n^\mbC(\varphi_1 \wedge \varphi_2)$
for every $n$ and hence $\lim_{n\to\infty}\delta_n^\mbC(\varphi_1 \wedge \varphi_2) = 1$.
\end{proof}

\noindent
Recall Definition~\ref{definition of l-colouring} about $l$-colourings viewed as functions.
The next step is to define, given an $l$-colourable structure $\mcA$ and
$l$-colouring $\gamma : A - \cl(\es) \to \{1, \ldots, l\}$, a formula which
describes which elements have the same colour (with respect to $\gamma$).

\begin{defi}\label{definition of zeta-gamma}
Suppose that $\mcA$ is an $l$-colourable $\Lrel$-structure with universe $A = \{a_1, \ldots, a_{\alpha}\}$
and that $\gamma : A - \cl(\es) \to \{1, \ldots, l\}$ 
is an $l$-colouring of $\mcA$. 
Then $\zeta_\gamma(x_1, \ldots, x_{\alpha})$ denotes the formula
\begin{align*}
\bigwedge_{\substack{ 1 \leq i \leq \alpha \\ a_i \in \cl(\es)}} \theta_0(x_i) \ \wedge \ 
\bigwedge_{\substack{ 1 \leq i,j \leq \alpha \\ a_i, a_j \notin \cl(\es) \\ \gamma(a_i) = \gamma(a_j) }} 
\xi(x_i,x_j) 
\ \wedge \ \ 
\bigwedge_{\substack{ 1 \leq i,j \leq \alpha \\ a_i, a_j \notin \cl(\es) \\ \gamma(a_i) \neq \gamma(a_j) }} 
\neg\xi(x_i,x_j).
\end{align*}
\end{defi}

\begin{rmk}~\label{ccrmk1}
Notice that if $\gamma : A - \cl(\es)$ and $\gamma' : A - \cl(\es)$ 
are such that for all $a,b \in A - \cl(\es)$, $\gamma(a) = \gamma(b) \Longleftrightarrow 
\gamma'(a) = \gamma'(b)$, 
then $\zeta_\gamma = \zeta_{\gamma'}$.
This is because $\xi$ only discerns which elements have the same colour and not which colour they have.
We will use this in Lemma~\ref{ccl1}.
\end{rmk}

\noindent
Recall Definition~\ref{chardef} of the characteristic formula of an
$l$-coloured or $l$-colourable structure,
with respect to an ordering of its universe.

\begin{defi}\label{definition of colour compatible extension axiom}
(i) For every $n \in \mbbN$, let $\eta_n(x_1, \ldots, x_n)$ denote the $L_{pre}$-formula
\[\forall y \bigg( \theta_n(x_1, \ldots, x_n, y) \ \rightarrow \ \bigvee_{i=1}^n y = x_i \bigg),\]
and note that, in any $l$-coloured, or $l$-colourable, structure, $\eta(x_1, \ldots, x_n)$ expresses
that $\{x_1, \ldots, x_n\}$ is a closed set.\\
(ii) Suppose that $\B$ is an $l$-colourable $\Lrel$-structure and that 
$\A \subsetneq \B$ is a closed substructure of $\B$.  
Let $A = \{a_1, \ldots, a_{\alpha}\}$ and $B = \{a_1, \ldots, a_{\beta}\}$, where $\beta > \alpha$.
Moreover, suppose that 
$\gamma'$ is an $l$-colouring of $\B$ which extends an $l$-colouring $\gamma$ of $\A$, 
so $\gamma' \uhrc A = \gamma$.
We call the following sentence an {\bf \em instance of the $l$-colour compatible
$\B/\A$-extension axiom}:
\begin{align*}
\forall x_1,\ldots,x_\alpha \exists x_{\alpha+1},\ldots,x_\beta 
\Big(&\big[\chi_\A(x_1,\ldots,x_\alpha) 
\ \wedge \ \zeta_\gamma(x_1,\ldots,x_\alpha) \ \wedge \ 
\eta_\alpha(x_1, \ldots, x_\alpha) \big] 
\\
\longrightarrow \ 
&\big[\chi_\B(x_1,\ldots,x_\beta) \ \wedge \ \zeta_{\gamma'}(x_1,\ldots,x_\beta) \ \wedge \ 
\eta_\beta(x_1, \ldots, x_\beta) \big] \Big). 
\end{align*}
There are only finitely many $l$-colourings of $\mcB$ and therefore 
there are only finitely many instances of the $l$-colour compatible $\B/\A$-extension axiom. We define the
{\bf \em $l$-colour compatible $\B/\A$-extension axiom} to be the conjunction of all these instances. 
A sentence $\varphi$ is called an {\bf \em $l$-colour compatible extension axiom} if it is the
$l$-colour compatible $\mcB/\mcA$-extension axiom for some closed substructure 
$\mcA \subset \mcB$ where $\mcB$ is
$l$-colourable.
\end{defi}

\noindent 
Observe that every $l$-colour compatible extension axiom is an $L_{rel}$-sentence (so none of the
symbols $P_1, \ldots, P_l$ occurs in it).
The next lemma shows that whenever $\mcA \subset \mcB$ are $l$-colour{\em able} structures
and $A$ is closed in $\mcB$,
then there is $k$ such that if $\mcM$ is an $l$-colour{\em ed} structure and has the 
$k$-extension property, then $\mcM$ satisfies the $l$-colour compatible $\mcB/\mcA$-extension axiom.
As a corollary we will then get that, with $\delta_n^\mbC$-probability approaching 1 as $n$ tends to infinity,
a random $\mcM \in \mbC_n$ satisfies the $l$-colour compatible $\mcB/\mcA$-extension axiom.

\begin{lma}~\label{ccl1}
Assume that $\B$ is an $l$-colourable $\Lrel$-structure and $\A \subsetneq \B$ is a closed substructure of $\B$. 
Let $k = \max(k_0, \dim(\mcB))$.
If $\mcM$ is an $l$-coloured structure and has the $k$-extension property, then
$\mcM$ satisfies the $l$-colour compatible $\mcB/\mcA$-extension axiom.
\end{lma}

\begin{proof}
Let $\mcB$ be $l$-colourable $\Lrel$-structures and let
$\mcA$ is a closed substructure of $\mcB$. Assume that 
$A = \{a_1, \ldots, a_{\alpha}\}$ and $B = \{a_1, \ldots, a_{\beta}\}$ where $\beta > \alpha$
and let $k$ be as in the lemma.
It is enough to prove that if $\mcM$ is an $l$-coloured structure with the $k$-extension property, then $\mcM$ 
satisfies each instance of the $l$-colour compatible $\B/\A$-extension axiom.
So assume that $\mcM$ is an $l$-coloured structure with the $k$-extension property and choose an arbitrary instance of the $l$-colour compatible $\B/A$-extension axiom, which uses a colouring $\gamma' : B - \cl(\es) \rightarrow \{1,\ldots,l\}$ and its restriction to a $A$, $\gamma = \gamma' \uhrc (A - \cl(\es))$. Assume that
\[\M\models \chi_\A(a'_1,\ldots,a'_\alpha) \ \wedge \ 
\zeta_\gamma(a'_1,\ldots,a'_\alpha) \ \wedge \ 
\eta_\alpha(a'_1,\ldots,a'_\alpha) \]
for some $a'_1,\ldots,a'_\alpha\in M$.
Let $\A' =\M\reduct\{a'_1,\ldots,a'_\alpha\}$. 
Then by the definition of $\chi_\A$ there is an isomorphism $f:\A'\reduct\Lrel \rightarrow \A$
such that $f(a'_i) = a_i$ for $i = 1, \ldots, \alpha$.
For all $a'_i \in \{a'_1, \ldots, a'_\alpha\} - \cl(\es)$ and $j \in \{1, \ldots, l\}$ let 
\[\gamma_0(a'_i) = j \hspace{0.5cm} \Longleftrightarrow \hspace{0.5 cm}\M\models P_j(a'_i),\]
so $\gamma_0$ is an $l$-colouring of $\A'\uhrc \Lrel$. 
By the choice of $k$ and since $\M$ has the $k$-extension property and $\M\models\zeta_\gamma(a'_1,\ldots,a'_\alpha)$ it follows (using~(2)) that, for all 
$a'_i, a'_j\in\{a'_1,\ldots,a'_\alpha\}-cl(\emptyset)$,
\[\gamma_0(a'_i) = \gamma_0(a'_j) 
\hspace{0.5 cm} \Longleftrightarrow\hspace{0.5cm}
\mcM \models \xi(a'_i, a'_j)
\hspace{0.5 cm} \Longleftrightarrow\hspace{0.5cm}
\gamma(a'_i)=\gamma(a'_j).\]
From this it follows that (by permuting the colours assigned by $\gamma'$ if necessary) we can find a $l$-colouring $\gamma_1'$ of $\B$ such that if $\gamma_1$ is the restriction of $\gamma_1'$ to $A$, 
then for all $a_i \in\{a_1,\ldots,a_\alpha\} - \cl(\emptyset)$ we have 
$\gamma_1(a_i) = \gamma_1(f(a'_i)) = \gamma_0(a'_i)$ and for all 
$a,b \in B - \cl(\emptyset)$,
\[\gamma_1'(a)=\gamma'_1(b) \hspace{0.5cm}\Longleftrightarrow \hspace{0.5cm}\gamma'(a)=\gamma'(b).\]
Now expand $\B$ into an $L$-structure $\B^+$ by adding colours to it according to $\gamma'_1$, that is, 
if $b\in B - \cl(\es)$ and $\gamma_1'(b) = i$ then let $\B^+\models P_i(b)$.
By the definition of $\B^+$ it follows that $f^{-1}$ is an $L$-isomorphism from $\B^+\reduct A$ onto $\M\reduct \{a'_1,\ldots,a'_\alpha\}$. 
Since $\M$ has the $k$-extension property, we may extend $f^{-1}$ into an embedding $g:\B^+\rightarrow\M$,
where the image of $g$ is a closed subset of $M$.
For every $i \in \{\alpha_1, \ldots, \beta\}$, let $a'_i = g(a_i)$.
Then 
\[\mcM \models \chi_{\mcB}(a'_1, \ldots, a'_\beta) \ \wedge \ \eta_\beta(a'_1, \ldots, a'_\beta).\]
Moreover, as $g$ is an $L$-embedding we have, for every $j \in \{1, \ldots, l\}$ and every 
$a \in B - \cl(\es)$, $\mcM \models P_j(g(a))$ $\Longleftrightarrow$ $\mcB^+ \models P_j(a)$.
By the choice of $\gamma'_1$ it follows that, for all $i, j \in \{1, \ldots, \beta\}$
such that $a'_i, a'_j \notin \cl(\es)$,
$\gamma'_1(a'_i) = \gamma'_1(a'_j)$ $\Longleftrightarrow$ $\gamma'(a_i) = \gamma'(a_j)$.
From the definition of $\zeta_{\gamma'}$ and the choice of $k$ it follows
that $\mcM \models \zeta_{\gamma'}(a'_1, \ldots, a'_\beta)$.
The chosen instance of the $l$-colour compatible $\B/\A$-extension axiom is hence satisfied by $\M$, 
and since it was an arbitrary instance, $\M$ has to satisfy all of the instances.
\end{proof}

\noindent 
Now we can prove that every $l$-colour compatible extension axiom will 
almost surely be satisfied in an $l$-colour{\em able} structure.

\begin{cor}~\label{colto1}
For every $l$-colour compatible extension axiom $\varphi$, 
$\lim_{n\rightarrow \infty} \delta_n^{\mbC} (\varphi) = 1$.
\end{cor}

\begin{proof}
Let $\varphi$ be an $l$-colour compatible extension axiom, so for some $\mcA \subseteq \mcB$
it is the $l$-colour compatible $\mcB / \mcA$-extension axiom. 
Let $k = \max(k_0, \dim(\mcB))$.
By Lemma~\ref{ccl1}, for every $n$, if $\mcM \in \mbK_n$ has the $k$-extension property 
then $\mcM \models \varphi$.
Hence, for every $n$,
\[\delta_n^{\mbK}\big(\{\mcM \in \mbK_n : \mcM \text{ has the $k$-extension property}\}\big) \ \leq \  \delta_n^{\mbK}(\varphi).\]
By Corollary~\ref{ccv1} we get $\lim_{n \to \infty} \delta_n^{\mbK}(\varphi) = 1$,
so it suffices to show that $\delta_n^{\mbC}(\varphi) = \delta_n^{\mbK}(\varphi)$ for all $n$.
But, as in the proof of Lemma~\ref{the property of U},
this follows from the definition of $\delta_n^{\mbC}$ and the fact that $\varphi$ is
an $\Lrel$-sentence.
\end{proof}

\noindent
The part $T_{ext}$ of the axiomatisation $T_\mbC$ consists, by definition, 
of all $l$-colour compatible extension axioms.
The axiomatisation $T_\mbC$ also needs to express that the formulas
$\theta_n$, $n \in \mbbN$, from Assumption~\ref{assumptions on languages} define
a pregeometry in every model of $T_\mbC$. This is the purpose of the
part of $T_\mbC$ which we denote $T_{pre}$.
More specifically, by using the formulas from Assumption~\ref{assumptions on languages}~(2), 
we can express, with an infinite set $T_{pre}$, of $L_{pre}$-sentences, 
properties~(1)--(3) of pregeometries in Definition~\ref{definition of pregeometry}
for {\em finite} sets.
In particular, since the closure of a set $A$ should not depend on how we order $A$, 
$T_{pre}$ contains, for each $n \geq 1$
and each permutation $\pi$ of $\{1, \ldots, n\}$, the sentence
\[ \forall x_1, \ldots, x_{n+1}\big( \theta_n(x_1, \ldots, x_n, x_{n+1} \big) \ \longleftrightarrow \ 
\theta_n(x_{\pi(1)}, \ldots, x_{\pi(n)}, x_{n+1}) \big). \]

\begin{lma}\label{every model of T-pre is a pregeometry}
Let $\mcM$ be an $L_{rel}$-structure such that $\mcM \models T_{pre}$.
Define a closure operator $\cl_\mcM$ as follows:
\begin{itemize}
\item[(a)] For every $n \in \mbbN$ and all $a_1, \ldots, a_{n+1} \in M$,
$a_{n+1} \in \cl_\mcM(a_1, \ldots, a_n)$ if and only if $\mcM \models \theta_n(a_1, \ldots, a_{n+1})$.
\item[(b)] For every $A \subseteq M$ and every $a \in M$, $a \in \cl_\mcM(A)$ if and only if
there is a finite $A' \subseteq A$ such that $a \in \cl_\mcM(A')$.
\end{itemize}
Then $(M, \cl_\mcM)$ is a pregeometry.
\end{lma}

\noindent
{\bf Proof.}
Suppose that $\mcM \models T_{pre}$ and let $\cl_{\mcM}$ be defined by~(a) and~(b).
From~(b) it follows that $\cl_\mcM$ has the finiteness property~(4) of 
Definition~\ref{definition of pregeometry} of a pregeometry.
From the definition of $T_{pre}$ and~(a) it follows that $\cl_\mcM$ has properties~(1)--(3)
of Definition~\ref{definition of pregeometry} of a pregeometry.
Hence, $(M, \cl_\mcM)$ is a pregeometry. 
\hfill $\square$
\\

\noindent
The fourth part of the aximatisation $T_\mbC$, denoted $T_{iso}$, will express that
``every closed finite substructure is $l$-colourable''.
For each $n\in\mathbb{N}$ let $\M_{n,1},\ldots,\M_{n,m_n}$ be an enumeration of all 
(finitely many) members of $\C_n$, and recall that $\chi_{\M_{n,i}}$ denotes the characteristic formula
of $\M_{n,i}$ (see Definition~\ref{chardef}). Recall that for every $n$, all structures in $\C_n$ have
the same universe (in fact their reduct to $L_{pre}$ is the same).
Let $s(n)$ be the cardinality of (the universe of) a structure in $\C_n$.
For every $n \in \mbbN$, there is an $L_{rel}$-sentence $\psi_n$ which expresses that 
every closed substructure of cardinality $s(n)$ is isomorphic to one of  $\M_{n,1},\ldots,\M_{n,m_n}$.
More precisely, we let $\psi_n$ be the $L_{rel}$-sentence
\begin{align*}
\forall x_1,\ldots,x_{s(n)}\Bigg(
&\bigg[ \bigwedge_{i \neq j} x_i \neq x_j \ \wedge \ 
\forall y \bigg( \theta_{s(n)}(x_1, \ldots, x_{s(n)}, y) \ \rightarrow \ 
\bigvee_{i = 1}^{s(n)} y = x_i \bigg) \bigg] \\
&\longrightarrow \ 
\bigvee_{i = 1}^{s(n)} \bigvee_{\pi} \chi_{M_{n,i}}(x_{\pi(1)}, \ldots, x_{\pi(s(n))}) \Bigg),
\end{align*}
where the disjunction `$\bigvee_{\pi}$' ranges over all permutations $\pi$ of $\{1,\ldots, s(n) \}$.
Let $T_{iso}=\{\psi_n : n \in \mbbN\}$.
Recall that $T_\xi = \{\varphi_1 \wedge \varphi_2\}$, where $\varphi_1$ and $\varphi_2$
were defined in Definition~\ref{definition of U}, that
$T_{ext}$ is the set of all $l$-colour compatible extension axioms and that
$T_{pre}$ was defined in the paragraph before Lemma~\ref{every model of T-pre is a pregeometry}.
Now we let
\[ T_\mbC \ = \ T_{pre} \cup T_{\xi} \cup T_{ext} \cup T_{iso}. \]
Notice that $T_\C$ contains only $\Lrel$-sentences. 

\begin{lma} ~\label{tclemma}
$T_\mbC$ is consistent and countably categorical, hence complete.
\end{lma}
 
\begin{proof}
From the definitions of $T_{pre}$ and $T_{iso}$ it follows that every $l$-colourable structure
is a model of $T_{pre} \cup T_{iso}$.
By compactness, Corollary~\ref{colto1}
and Lemma~\ref{the property of U}, it follows that $T_\C$ is consistent.

We now prove that $T_\mbC$ is countably categorical.
Assume that $\M$ and $\M'$ are $\Lrel$-structures such that $\M \models T_\C$, $\M'\models T_\C$ 
and $|M | = |M' | = \aleph_0$. 
Since $\mcM, \mcM' \models T_{pre}$ it follows from Lemma~\ref{every model of T-pre is a pregeometry} 
that if $\cl_\mcM$ is defined by saying
that $a \in \cl_\mcM(A)$ if and only if there are $m$ and $b_1, \ldots, b_m \in A$ such that
$\mcM \models \theta_m(b_1, \ldots, b_m, a)$, then $(M, \cl_\mcM)$ is a pregeometry;
and similarly for $\mcM'$.
By a back and forth argument we will build partial isomorphisms between $\M$ and $\M'$ such that each new one
extends the former ones. The union of these partial isomorphisms shows that $\M \cong \M'$.
The main part of the argument is to prove the following:

\begin{claim} \it
Let $\A \subseteq \M$, $A' \subseteq\M'$ be finite closed substructures 
(so $\cl_{\mcM}(\es) \subseteq A$ and $\cl_{\mcM'}(\es) \subseteq A'$)
and suppose that there is an isomorphism $f:\A\rightarrow\A'$ such that for all $a,b \in A - \cl_{\mcM}(\es)$ 
we have $\M \models \xi(a,b) \Longleftrightarrow \M' \models \xi(f(a),f(b))$. 
Then for every $c\in M - A$ ($d \in M' - A'$) there exist 
a closed subtructure $\B' \subseteq \M'$ ($\B \subseteq \M$)
and an isomorphism $g:\M \uhrc \cl_\M(A\cup \{c\}) \rightarrow \B'$ 
($g : \B \rightarrow \M' \uhrc \cl_{\M'}(A \cup \{d\}$) such that 
$\mcA' \subseteq \mcB'$ ($\mcA \subseteq \mcB$), 
$g$ extends $f$ and for all $a,b \in \cl_\M(A\cup \{c\}) - \cl_{\mcM}(\es)$ 
($a, b \in B - \cl_{\mcM}(\es)$) we have 
$\M \models \xi(a,b) \Longleftrightarrow \M' \models \xi(g(a),g(b))$
\end{claim}

\begin{proof}[Proof of the claim.]
Let $\A \subseteq \M$, $A' \subseteq\M'$ be finite closed substructures
and $f:\A\rightarrow\A'$ an isomorphism such that for all $a,b \in A - \cl_{\mcM}(\es)$ 
we have $\M \models \xi(a,b) \Longleftrightarrow \M' \models \xi(f(a),f(b))$. 
Let $A = \{a_1, \ldots, a_{\alpha}\}$ and $A' = \{f(a_1), \ldots, f(a_{\alpha})\}$.
Suppose that $c \in M - A$ and let $B = \cl_{\mcM}(A \cup \{c\})$.
(The case when $d \in M' - A'$ is proved in the same way.)
Enumerate $B$ as $B = \{b_1, \ldots, b_{\beta}\}$ so that $b_i = a_i$ for $1 \leq i \leq \alpha$.
Note that we must have $\beta > \alpha$.
Let $\B = \M \uhrc B$. 

Since $\mcM \models T_\mbC$ we have $\mcM \models \varphi_1 \wedge \varphi_2$,
which implies that
\begin{itemize}
\item[(i)] $\xi(x,y)$ defines an equivalence relation on $M - \cl_\mcM(\es)$ 
which we denote $\sim_{\xi}$,
\item[(ii)] there is a substructure of $\mcM$ which is isomorphic to $\mcU$
(from Definition~\ref{definition of U}), and for notational
simplicity we denote it by $\mcU$, so $\mcU \subseteq \mcM$, and there are 
$u_1, \ldots, u_l \in U - \cl_\mcM(\es)$ such that if $1 \leq i < j \leq l$ then 
$u_i \not\sim_{\xi} u_j$, and
\item[(iii)] for all $a \in M - \cl_{\mcM}(\es)$ there is $j \in \{1, \ldots, l\}$ such that 
$a \sim_{\xi} u_j$.
\end{itemize}

\noindent
It follows that the restriction of $\sim_{\xi}$ to $(U \cup B) - \cl_\mcM(\es)$ has exactly $l$ 
equivalence classes.
By assumption, $\xi(x,y)$ is an existential formula.
As $U \cup B$ is finite, it follows that there is a {\em finite} closed substructure $\mcN \subseteq \mcM$
such that $U \cup B \subseteq N$ and for all $a, b \in U \cup B$ we have $\mcM \models \xi(a,b)$ if and only
if $\mcN \models \xi(a,b)$. Consequently:
\[\text{for all } a, b \in (U \cup B) - \cl_\mcM(\es), \ a \sim_{\xi} b 
\ \Longleftrightarrow \ \mcN \models \xi(a,b).\]

\noindent
Since $\M\models T_{iso}$ we know that $\mcN$ is an $l$-colourable structure.
Let $\gamma : N - \cl_{\mcN}(\es) \to \{1, \ldots, l\}$ be an $l$-colouring of $\mcB$. 
Define an equivalence relation $\sim_\gamma$ on $N - \cl_\mcN(\es)$ by 
\[a \sim_\gamma b \ \Longleftrightarrow \ \gamma(a) = \gamma(b) \ \text{ for all } a,b \in N - \cl_\mcN(\es).\]  
Since $\mcU$ cannot (by its definition) be $l'$-coloured if $l' < l$, it follows that
$\sim_\gamma$ has exactly $l$ equivalence classes. In fact, the restriction of $\sim_\gamma$
to $(U \cup B) - \cl_\mcN(\es)$ has exactly $l$ equivalence classes.

Let $a, b \in (U \cup B) - \cl_\mcM(\es)$ and suppose that $a \sim_{\xi} b$.
Then $\mcM \models \xi(a,b)$ and by the choice of $\mcN$ we also get $\mcN \models \xi(a,b)$.
By~(1) we must have $\gamma(a) = \gamma(b)$, so $a \sim_\gamma b$.
Hence, the restriction of $\sim_{\xi}$ to 
$(U \cup B) - \cl_\mcM(\es)$ is a refinement of the restriction of $\sim_\gamma$ to 
$(U \cup B) - \cl_\mcM(\es)$.
As both equivalence relations restricted to $(U \cup B) - \cl_\mcM(\es)$ have
exactly $l$-equivalence classes, it follows that the must coincide.
In other words, for all $a,b \in (U \cup B) - \cl_\mcM(\es)$ 
we have $a \sim_{\xi} b$ if and only if $a \sim_\gamma b$, so
\begin{itemize}
\item[(iv)] for all $a,b \in B - \cl_\mcM(\es), \ \gamma(a) = \gamma(b) 
\ \Longleftrightarrow \ \mcM \models \xi(a,b)$.
\end{itemize}
Let $\gamma_B = \gamma \uhrc B$ and $\gamma_A = \gamma \uhrc A$.
Since $\chi_\mcA(x_1, \ldots, x_\alpha)$ is the characteristic formula of $\mcA$ we have
$\mcM \models \chi_\mcA(a_1, \ldots, a_\alpha)$.
From~(iv) it follows that $\mcM \models \zeta_{\gamma_A}(a_1, \ldots, a_\alpha)$
(see Definition~\ref{definition of zeta-gamma}).
The assumptions of the claim now imply that
\[\mcM' \models \chi_\mcA(f(a_1), \ldots, f(a_\alpha)) \ \wedge \ 
\zeta_{\gamma_A}(f(a_1), \ldots, f(a_\alpha)) \ \wedge \ 
\eta_\alpha(f(a_1), \ldots, f(a_\alpha)).\]
Since $\mcM' \models T_{ext}$ it follows, in particular, that $\mcM'$ satisfies the following instance
of the $l$-colour compatible $\mcB/\mcA$-extension axiom:
\begin{align*}
\forall x_1, \ldots, x_\alpha \exists y_1, \ldots, y_\beta \Big(
&\big[\chi_\mcA(x_1, \ldots, x_\alpha) \ \wedge \ \zeta_{\gamma_A}(x_1, \ldots, x_\alpha) 
\ \wedge \ \eta_\alpha(x_1, \ldots, x_\alpha) \big] \\
\longrightarrow \
&\big[ \chi_\mcB(x_1, \ldots, x_\beta) \ \wedge \ \zeta_{\gamma_B}(x_1, \ldots, x_\beta) 
\ \wedge \ \eta_\beta(x_1, \ldots, x_\beta) \big]\Big).
\end{align*}
It follows that there are closed substructure $\B' \subseteq \M'$ such that $\A' \subseteq B'$ 
and an isomorphism $g : \mcB \to \mcB'$ which extends $f$
with the property that for all $a, b \in B - \cl(\es) = \cl_{\mcM}(A \cup \{c\}) - \cl_\mcM(\es)$,
$\mcM \models \xi(a,b)$ if and only if $\mcM' \models \xi(g(a),g(b))$.
\end{proof}

\noindent {\it Continuation of the proof of the lemma.} 
By the claim, it now suffices to prove that the assumptions of the claim hold for at least
one pair, $\mcA, \mcA'$, such that $\mcA \subseteq \mcM$ and $\mcA' \subseteq \mcM'$ are
closed substructures of the respective superstructure.
We claim that this holds for $\mcA = \mcM \uhrc \cl_\mcM(\es)$ and $\mcA' = \mcM' \uhrc \cl_{\mcM'}(\es)$.
Indeed, by the definition of $l$-colourable structure,
Assumption~\ref{assumptions on languages} and $\mcM, \mcM' \models T_{iso}$, it follows that
with $\mcA$ and $\mcA'$ defined in this way we have 
$\mcA \cong \mcA'$.
Since $A - \cl_\mcM(\es) = \es$ the preservation of $\xi$ for $a,b \in A - \cl_\mcM(\es)$
is trivially satisfied. 
This concludes the proof of Lemma~\ref{tclemma}, and hence also of Theorem~\ref{main theorem, general form}.
\end{proof}

\noindent
{\bf Errata for \cite{Kop12}.}
Here we mention some missed assumptions that should be added to some results of~\cite{Kop12}.\\
(i) Theorems~7.31,~7.32 and~7.34 in \cite{Kop12} needs part~(5) from 
Assumption~\ref{assumptions on languages}
in this article. (This assumption could most conveniently be added to Assumption~7.10 in~\cite{Kop12}.)
This assumption is implicitly used in the proof of Lemma~8.5 in \cite{Kop12} and guaratees that 
if $\mcA$ is a closed substructure of a represented structure (see Definition~7.4 \cite{Kop12}) $\mcM$
and $B \subseteq A$, then the closure of $B$ is the same whether computed in $\mcA$ or in $\mcM$. \\
(ii) Theorem~7.32 in \cite{Kop12} needs the following two additional assumptions 
(implicitly made in the proof of
that theorem):
\begin{itemize}
\item[(a)] There is, up to isomorphism, a unique represented structure (Definition~7.4 in~\cite{Kop12})
with dimension 0. (Note that by the definitions in this article, there is a unique, up to isomorphism,
(strongly) $l$-coloured structure with dimension 0.) 
This assumption is needed in the proof of Lemma~8.12 in~\cite{Kop12}.
Without it, one only gets a limit law (convergence, but not necessarily to 0 or 1), 
since one gets a distinct ``limit theory'' for every represented isomorphism type of dimension 0.

\item[(b)] For every $n \in \mbbN$, there is a ``characteristic'' $L_{pre}$-formula 
$\chi_{\mcG_n}(x_1, \ldots,x_{m_n})$ of $\mcG_n$,
where $m_n = |G_n|$, such that if $\mcA$ is an $L_{pre}$-structure in which the formulas $\theta_n$
define a pregeometry and $\mcA \models \chi_{\mcG_n}(a_1, \ldots, a_s)$ for some enumeration $a_1, \ldots, a_s$ of $A$,
then $\mcA \cong \mcG_n$. This is~(7) of Assumption~\ref{assumptions on languages} in this article
(except for omitting here the requirement that $\chi_{\mcG_n}$ is quantifier-free),
and it holds for the examples of pregeometries (and corresponding languages) 
considered in \cite{Kop12} and in this article.
But it is not a consequence of the other assumptions made (in Assumption~7.3 and~Assumption~7.10
in~\cite{Kop12}), so it needs to be added.
\end{itemize}
These remarks affect {\em only} Sections~7--8 in~\cite{Kop12}, because the other sections
do {\em not} consider (nontrivial) pregemetries.

\end{document}